\documentclass[12pt]{amsart}
\usepackage{epsfig,color}
\usepackage{amssymb}
\usepackage{caption}
\usepackage[utf8]{inputenc}
\usepackage[T1]{fontenc}
\usepackage{enumerate}
\usepackage{cite}
\usepackage{comment}
\usepackage{graphicx}
\newtheorem{Thm}[equation]{Theorem}
\newtheorem{Cor}[equation]{Corollary}
\newtheorem{Lem}[equation]{Lemma}
\newtheorem{Pro}[equation]{Proposition}
\newtheorem{Que}[equation]{Question}

\theoremstyle{definition}

\newtheorem{Def}[equation]{Definition}

\theoremstyle{remark}

\newtheorem{Rem}[equation]{Remark}

\newtheorem{claim}[equation]{Claim}

\numberwithin{equation}{section}
\makeatletter
\renewcommand{\c@figure}{\c@equation}
\makeatother

\newcommand{\eop}[1]{{\flushright\hfill\fbox{\bf #1}}}

\newcommand{\ack}{\noindent{\bf Acknowledgement.}}

\begin{document}

\title{Rigidity of flat holonomies}

\author{G\'erard Besson}
\address{CNRS, Universit\'e Grenoble Alpes\\ 
Institut Fourier, CS 40700\\
38058 Grenoble C\'edex 09, France}
\urladdr{http://www-fourier.ujf-grenoble.fr/~besson}
\email{g.besson@univ-grenoble-alpes.fr}

\author{Gilles Courtois}
\address{CNRS, Institut de Math\'ematiques de Jussieu-Paris Rive Gauche, UMR 7586\\Sorbonne Universit\'e, UPMC Univ Paris 06,
Univ Paris Diderot, Sorbonne Paris Cit\'e, F-75005, Paris, France}
\urladdr{https://webusers.imj-prg.fr/~gilles.courtois}
\email{gilles.courtois@imj-prg.fr}

\author{Sa'ar Hersonsky}
\address{Department of Mathematics\\ 
University of Georgia\\ 
Athens, GA 30602, USA}
\urladdr{http://www.math.uga.edu/~saarh}
\email{saarh@uga.edu}

\thanks{}
\date{\today}

\begin{abstract}

We prove that the existence of one horosphere in the universal cover of a closed
Riemannian manifold of dimension $n \geq 3$ with strongly $1/4$-pinched or relatively $1/2$-pinched sectional curvature, on which %coupled with a geometric-dynamical property of 
the stable holonomy along one horosphere coincide with the Riemannian parallel transport,  implies that the manifold is homothetic to a real hyperbolic manifold. 
\end{abstract}

\maketitle

\section{Introduction}
Mostow's seminal rigidity theorem  \cite{Mos}  asserts that the geometry of a closed hyperbolic manifold of dimension greater than two is determined by its fundamental group. Inspired by Mostow's theorem, we undertake a study of related, yet, more general themes. In this paper, 
%we will first relax the  assumption on the sectional curvature in Mostow's theorem  and allow it to be quarter negatively curved pinched. Moreover, 
we look at natural geometric submanifolds, the horospheres, and ask to what extent do these determine the geometry of the whole manifold.   Precisely, we are concerned  with the following general question:

\begin{Que}
\label{qu:Qu1}
Does the geometry of the horospheres of a closed, negatively curved manifold of dimension greater than two, determine the geometry of the whole manifold?
\end{Que}

In general, there are very few answers to Question~\ref{qu:Qu1}, and all of these relate the {\it extrinsic} geometry of the horospheres to the geometry of $M$. For instance, by combining \cite{FoL} and \cite{BCG1} (see \cite{BCG1}, Corollary 9.18) one shows that if all the horospheres have constant mean curvature, then the underlying  manifold is locally symmetric (of negative curvature).  Let us recall that the mean curvature of a hypersurface is related to the derivative of its volume element in the normal direction to the hypersurface, and hence the mean curvature is an extrinsic quantity. In this paper, our main hypothesis is to relax the  assumption on the sectional curvature in Mostow's theorem  and allow it to be strictly quarter negatively curved pinched. In this case constant mean curvature of the horospheres only occur for real hyperbolic manifolds (up to homothety). In contrast,  we would like to  emphasize that we only  consider the intrinsic properties of the induced metric on the horospheres.

\smallskip

Before stating our main theorem, let us recall a few important features of the manifolds under consideration and  results that are related to our work in this paper.  Let $M$ denote an $(n+1)$-dimensional, closed, Riemannian manifold  endowed with a metric of negative sectional curvature, $n\geq 2$. It follows from the Cartan-Hadamard theorem that  $\tilde M$, the universal cover of $M$, is diffeomorphic to $\mathbb{R}^{n+1}$. Let $\tilde M$ be endowed with the pull-back Riemannian metric from $M$, under the natural projection $\pi: {\tilde M} \to M$. The geometric boundary $\partial\tilde M$ of $\tilde M$,  is the set of equivalence classes of geodesic rays in $\tilde M$, where two geodesic rays are equivalent if they remain at a bounded Hausdorff distance. We recall that, in our context, it is homeomorphic to $\mathbb{S}^n$.

\smallskip

Given a point, $x_0\in \tilde M$, and a unit tangent vector, $\tilde v\in T_{x_0} \tilde M$, we let  $c_{\tilde v}$  denote the unique geodesic ray determined by  $c_{\tilde v}(0)=x_0$ and $\dot c _{\tilde v}(0) =\tilde v$. It is well known that the map, $\tilde v \in T_{x_0} \tilde M \mapsto [c_{\tilde  v}] \in \partial \tilde M$, defines a homeomorphism between the unit sphere in $T_{x_0} \tilde M$ and $\partial \tilde M$. Given a point $\xi = [c_{\tilde v}] \in \partial \tilde M$, the Busemann function 
$B_\xi(\cdot)$ is then defined for all $\xi \in \partial \tilde M$ and for all $x \in \tilde M$, by $B_{\xi}(x)= \lim_{t\to \infty} (d(x,c_{\tilde v} (t)) - d(x_0, c_{\tilde v} (t)))$. 

\smallskip

Since $M$ is a closed negatively curved manifold,  for each $\xi \in \partial \tilde M$ it is known that 
the Busemann function $B_\xi(\cdot)$ is $C^\infty$-smooth.  Furthermore, 
for any $t\in\mathbb{R}$,
 the level set  $$H_{\xi}(t) =\left\{ x\in {\tilde M};\, B_\xi(x)= t\right\}$$ is a smooth
submanifold of $\tilde M$ which is diffeomorphic to $\mathbb{R}^{n}$ and which is  called a {\it horosphere} centred at $\xi$. The sublevel set $$HB_{\xi}(t) =\left\{ x\in {\tilde M};\, B_\xi(x)\leq  t\right\}$$ is called a {\it horoball}.
It follows that  horospheres inherit a complete Riemannian metric induced by the restriction of the metric of $\tilde M$. For instance, if 
$(M,g)$ is a real hyperbolic manifold,
every horosphere of $\tilde M$ is flat and therefore isometric to the Euclidean space $\mathbb R^{n}$.

\smallskip

So far we defined horospheres as special submanifolds in $\tilde M$. However,  a dynamical perspective turns out to be important in the proof of the main theorem. Let $\tilde p : T^1 \tilde M \rightarrow \tilde M$ and $p : T^1M \rightarrow M$ denote  the natural projections. The geodesic flow $\tilde g_t$ on $T^1\tilde M$ is known to be an Anosov flow, that is, the tangent bundle  $TT^1 \tilde M$ admits a decomposition as $TT^1\tilde M = \mathbb R X \oplus \tilde E^{ss} \oplus \tilde E^{su}$, where $X$ is the vector field generating the geodesic flow and
$\tilde E^{ss}$, $\tilde E^{su}$ are the strong stable and strong unstable distributions, respectively. These distributions are known to be integrable, invariant under the differential $d \tilde g_t$ of the geodesic flow, and to give rise to two transverse foliations of $T^1\tilde M$, $\tilde W^{ss}$ and $\tilde W^{su}$, the strong stable and strong unstable foliations, respectively, whose leaves are smooth submanifolds. A classical property of these foliations is that in general they are transversally H\" older with exponent less than one, and when the sectional curvature, denoted by K is strictly $1/4$-pinched (i.e., $-4<K\leq-1$), they are transversally $C^1$ (see  \cite[page 226]{Hi-Pu-Sh}), but we do not use such a regularity.

\smallskip

A link between the two point of views on horospheres is the following. For $\tilde v\in T^1\tilde M$, the strong stable leaf $\tilde W^{ss}(\tilde v)$ through $\tilde v$ is defined to be  the set of unit vectors 
$\tilde w\in T^1\tilde M$ which are normal to the horosphere $H_\xi (t)$ and pointing inward the horoball
$HB_\xi (t)$ in the direction of $\xi = c_{\tilde v}(+\infty)$, with $t=B_\xi(\tilde p({\tilde v}))$ so that $H_\xi (t) = \tilde p (W^{ss} ({\tilde v}))$. 
\smallskip

\begin{comment}
BUTLERBUTLERBUTLER
Recently, C. Butler \cite{B}  obtained a deep  characterization of closed, real hyperbolic manifolds. His result which  is  essential to the proof of our main theorem and is related to the way the geometry of horospheres evolves under the action of the geodesic flow. Butler showed, in what might be called now as {\it Lyapunov rigidity},  that the equality of the modulus of the eigenvalues of $d  g _t | E^{ss}(v)$ along {\it every} periodic geodesic has an important geometric consequence. Let us recall his theorem: 
\begin{Thm}[\cite{B}, Theorem 1.1]
\label{butler}
Let $M$ be a closed, negatively curved manifold of dimension $n\geq 3$.  For
 a periodic orbit $g _t(v)$ of the geodesic flow on $T^1 M$ with period $l(v)$, let $\xi _1 (v), \dots, \xi_{n}(v)$ be the complex eigenvalues
of $D g _{l(v)}(v)| E^{ss}(v)$, counted with multiplicities. Assume that
$|\xi _1 (v)|= \dots =|\xi_{n}(v)|$ hold for each periodic orbit $g _t(v)$, then $M$ is homothetic to a compact quotient of the real hyperbolic space. 
\end{Thm}
\end{comment}

With this notation in place, let us now describe our main theorem and the foundational work we build upon.
In Section~\ref{connection}, we will recall the construction of the {\sl stable holonomy}. The notion of stable holonomy goes back to the work of Bonatti-Mont-Viana,\cite{BMV}, and has been extensively studied by various authors, M. Viana, \cite{Viana} (also in the non-uniformly hyperbolic setting), Avila-Viana, \cite{AV}, Santamaria-Viana, \cite{ASV}, Kalinin-Sadovskaya, \cite{KS}, in the context of partially hyperbolic systems. In our setting, it is a family, for each horosphere, of isomorphisms between the tangent spaces at
%$\{ \Pi ^\xi (x,y) \}_{x,y,\xi}$ between the tangent spaces at
any two points of it. Given $\xi \in \partial \tilde M$ and $x,y$ a pair of points on
a horosphere $H_\xi$ centered at $\xi$, we will informally denote $\Pi ^\xi (x,y)$ the isomorphism
between the tangent spaces to this horosphere at $x$ and $y$ and the stable holonomy will be the collection of all these 
isomorphisms $\Pi ^\xi (x,y)$.   
%, for $\xi \in \partial \tilde M$ and $s\in \mathbb R$. 
The stable holonomy
was originally constructed as a family, for each strong stable leaf of the geodesic flow, of isomorphisms $\mathcal H(\tilde v, \tilde w)$ between the tangent spaces to this leaf at any pair of points $\tilde v, \tilde w$ and not on the horospheres as we will present here. However, the two constructions 
are equivalent since there is the conjugation 
$\mathcal H(\tilde v, \tilde w) =  D\tilde p (\tilde w)^{-1} \circ \Pi ^\xi (x,y) \circ D\tilde p (\tilde v)$, where $\tilde p \tilde v=x, \tilde p \tilde w=y$
and $c_{\tilde v} (+\infty) = c_{\tilde w} (+\infty)=\xi$ (see more in the appendix).  This construction, which holds in the general setting of linear cocycles over partially hyperbolic diffeomorphisms, requires the 'fiber bunched' condition, as is
\cite{KS} (more on this in section~\ref{connection}). In the context of the geodesic flow of a negatively curved closed manifold, the fiber bunched condition is a consequence of a pinching condition on the sectional curvature. We will consider two kinds of pinching. The strong $1/4$-pinching of the curvature means that for every $x\in M$, the sectional curvature $K(x)$ satisfies
\begin{equation}\label{strongpinching}
-4<K(x) \leq-1.
\end{equation}
Given $a>0$, the curvature of $M$ is said to be relatively $a$-pinched if there exists a strictly negative function 
$C : M \to \mathbb R_{<0}$ such that for every $x\in M$, the sectional curvature satisfies 
\begin{equation}\label{relativepinching}
C(x) \leq  K(x) < aC(x).
\end{equation}  
In general, none of these two pinching conditions imply the other.
To the best of our knowledge, a stable holonomy cannot be defined without {\it some} pinching condition on the curvature.
In the sequel, we will describe the construction of the stable holonomy on the horospheres under the strong $1/4$-pinching or
the relative $1/2$-conditions.
On the other hand, every horosphere $H_\xi (s)$ carries the Riemannian metric induced by the one of $\tilde M$. In particular,
for every pair of sufficiently close points $x,y \in H_\xi (s)$, there is a unique minimizing geodesic of $H_\xi (s)$ joining them.
Hence, we  may consider the parallel transport associated to the Levi-Civita connecion of the induced 
metric on $H_\xi (s)$, denoted by $P^\xi _s (x,y)$, between the tangent spaces to
$H_\xi (s)$ at these points $x$ and $y$.
As mentioned before, in the case of $K\equiv -1$, the induced Riemannian metric on horospheres is flat
and the stable holonomy $\Pi ^\xi _s (x,y)$ and the parallel transport $P^\xi _s (x,y)$ coincide for every
pairs of points $x$ and $y$ on $H_\xi (s)$. Our main result is that the {\it converse} is true among 
strongly $1/4$-pinched or $1/2$-relatively pinched negatively curved manifolds.

\begin{Thm}[Main Theorem]
\label{main-thm}   
Let $M$ be a closed, Riemannian manifold of dimension $n\geq 3$, endowed with a strongly $1/4$-pinched or $1/2$-relatively pinched negatively curved sectional curvature.  Assume that there exists 
$\xi \in \partial \tilde M$ and $s\in\mathbb{R}$ such that for every pair of points $x,y \in H_\xi (s)$ joined by a unique minimizing 
geodesic, the stable holonomy $\Pi ^\xi _s (x,y)$ is identical to the parallel transport $P_s ^\xi (x,y)$. Then $(M,g)$ is homothetic to a real hyperbolic manifold.
\end{Thm}

As mentioned before, the restriction on the sectional curvature ensures the existence of the stable holonomy. 
%It is a natural assumption, however, the question of whether the above theorem is still true without it remains open. 
For Theorem \ref{main-thm} to hold, it is indeed sufficient to make the assumption for a {\it single} horosphere in $\tilde M$ since in Proposition \ref{one-all}  we show that it implies that {\it all} horospheres satisfy it. Since it is known that pinched curvature implies pinching of Lyapunov exponents we could hope, as suggested by the anonymous referee, that a pinching of Lyapunov exponents might be sufficient. This point is left for further study.

\smallskip

In the case that $\dim M = 2$, Theorem \ref{main-thm}, may still be true. However, 
our proof in the case $\dim M \geq 3$ does not apply since it relies on Theorems \ref{butler} and \ref{G_A-hyperbolic} which both require our assumption on the dimension, see more details below. 

\smallskip
Essential to the proof of our main theorem is the following deep characterization of closed, real hyperbolic manifolds stated by Butler \cite{B}. This result is related to the way the geometry of horospheres evolves under the action of the geodesic flow. Butler showed, in what might be called now as {\it Lyapunov rigidity},  that the equality of the modulus of the eigenvalues of $d  g _t | E^{ss}(v)$ along {\it every} periodic geodesic has an important geometric consequence. Let us recall his theorem: 
\begin{Thm}[\cite{B}, Theorem 1.1]
\label{butler}
Let $M$ be a closed, negatively curved manifold of dimension $n\geq 3$.  For
 a periodic orbit $g _t(v)$ of the geodesic flow on $T^1 M$ with period $l(v)$, let $\xi _1 (v), \dots, \xi_{n}(v)$ be the complex eigenvalues
of $D g _{l(v)}(v)| E^{ss}(v)$, counted with multiplicities. Assume that
$|\xi _1 (v)|= \dots =|\xi_{n}(v)|$ hold for each periodic orbit $g _t(v)$, then $M$ is homothetic to a compact quotient of the real hyperbolic space. 
\end{Thm}
In this theorem, the assumption on $\dim M \geq 3$ is indeed necessary. Indeed, let us consider a closed surface $M$ with 
a  $1/4$-pinched negative sectional curvature Riemannian metric $g$.
The metric $g$ can be chosen to be, for example, a small perturbation of an hyperbolic metric. In this case, the horospheres in $\tilde M$ endowed with their induced metric are complete Riemannian lines and the assumption on the eigenvalues of $D g _{l(v)}(v)| E^{ss}(v)$
along periodic orbits $g_t (v)$ does not provide any useful information; indeed there is a single eigenvalue and the action of $Dg _t$ on $E^{ss}$ is therefore trivially conformal. 

\medskip

Theorem \ref{main-thm} is a consequence of Theorem \ref{butler}, Proposition~\ref{one-all}, and the following result.
\begin{Thm}\label{eigenvalues} Under the assumptions of Theorem~\ref{main-thm}, let $c_{\tilde v} (t)$ projects to a periodic geodesic $c_v (t)$ of period $l(v)$ in $M$ and let $\xi = c_{\tilde v} (+\infty)$.  
Then, the complex eigenvalues of $Dg _{l(v)}(v)|E^{ss}(v)$ 
satisfy $|\xi _1 (v)|= \dots =|\xi_{n}(v)|$. \end{Thm}

\begin{comment}
The idea behind this Theorem is the following: a consequence of the coincidence of the 
the parallel transport of the Levi-Civita connexion and the stable holonomy 
on the horospheres is that the Levi-Civita connexion is {\it flat} and {\it invariant under the geodesic flow}. As mentioned before,
these two properties are true in constant negative curvature and in the case of the Heintze groups for one family
of horospheres. These properties 
\end{comment}
Let us now briefly describe the proof of Theorem \ref{eigenvalues}.  First note that the closeness of the manifold of $M$ is a necessary assumption as one can verify on the examples given by the {\em Heintze groups}.
Recall that a Heintze group is a solvable group  $G_A := \mathbb R \ltimes_A \mathbb R^n$, where $A$ is an $n\times n$ real matrix and $\mathbb R$ acts on
$\mathbb R ^n$ by $x\rightarrow e^{tA} x$. In the case that the real parts of the eigenvalues of $A$ have the same sign, Heintze \cite{He1} showed the existence of left invariant metrics on $G_A$ with negative sectional curvature. In this case, horospheres 
centered at a particular point on $\partial G_A$ and endowed with the induced metric are flat (see section~\ref{se:horo} and in particular (\ref{modele})). Notice that for a Heintze group, the existence of one 'flat' horosphere does not imply that all horospheres are flat. Indeed, crucial in the proof of Proposition \ref{one-all} is the fact that the metric under consideration comes from a closed Riemannian manifold while a Heintze group does not have any cocompact quotient unless it is the hyperbolic space. If $A$ is a multiple of the identity matrix, $G_A$ is then homothetic to the real hyperbolic space; furthermore, it was proved by Heintze in \cite{He} that the Heintze groups $G_A$ have no cocompact lattice unless they are homothetic to the hyperbolic space.
Moreover, X. Xie obtained a necessary condition for $G_A$ to be quasi-isometric to a finitely-generated group. His result is  also essential for the proof of our main Theorem. 
Before stating it, recall that, given an $n\times n$-matrix $A$, the 'real part Jordan form' of $A$ is obtained from the Jordan form of $A$ by replacing each diagonal entry with its real part and reordering to make it canonical.

\smallskip

\begin{Thm}[\cite{Xie1}, Corollary 1.6]
\label{G_A-hyperbolic}
Let $A$ be an $n\times n$ real matrix whose eigenvalues all have positive real parts. If $G_A$ is quasi-isometric to a finitely generated group, then the real part Jordan form of $A$ is a multiple of the identity matrix.
\end{Thm}

\smallskip

The main idea of the proof of Theorem~\ref{eigenvalues} is therefore to show that for each 
periodic orbit $g_t (v)$ of the geodesic flow of $T^1 M$ of period $l(v)$,  $\tilde M$ is quasi-isometric to a Heintze group $G_A$, where $A$ is a matrix whose eigenvalues {\it all} have positive real parts and 
such that $e^{l(v)A}$ is conjugate to $Dg _{l(v)}(v)|E^{ss}(v)$. By assumption, $M$ is a closed manifold endowed with a negatively curved metric. It is well known that  $\tilde M$ is quasi-isometric to the fundamental group of $M$ which is, in particular,  finitely-generated. Hence, $G_A$ turns out to be quasi-isometric to a finitely-generated group. It now follows from the above mentioned theorem of Xie that the real part of the eigenvalues of $A$ coincide and therefore, the eigenvalues of $Dg _{l(v)}(v)|E^{ss}(v)$ have the same modulus.

\smallskip

Therefore, we are left with  proving  that $\tilde M$ is quasi-isometric to a Heintze group $G_A$. This is done as follows.
 Let us fix a geodesic in $\tilde M$ with an endpoint  $\xi\in\partial\tilde M$. 
The set of stable horospheres $H_\xi (t)$ centered at $\xi$ and the set of geodesics asymptotic to $\xi$ define two orthogonal foliations 
of $\tilde M$. These foliations determine horospherical coordinates $\mathbb R \times H_\xi (0) = \mathbb R \times \mathbb R ^n$ on $\tilde M$. In these coordinates, 
the metric of $\tilde M$ decomposes at every point $(t,x) \in \mathbb R \times \mathbb R^n$ as an orthogonal sum 
\begin{equation}\label{def-tildeg}
\tilde g = dt ^2 + h_t ,
\end{equation}
where $dt ^2$ is the standard metric on $\mathbb R$ and $h_t$ is a one parameter family of flat metrics on $H_\xi (0) = \mathbb R^n$. 
On the other hand, a Heintze group $G_A$ is, by definition, also diffeomorphic to $\mathbb R \times \mathbb R ^n$ with a metric, written similarly  at every point $(t,x) \in \mathbb R \times \mathbb R^n$, as the orthogonal sum

\begin{equation}\label{def-g_A}
g_A := dt ^2 + < e^{tA} \cdot, e^{tA} \cdot >,
\end{equation}
where $< e^{tA} \cdot, e^{tA} \cdot >$ is a one parameter family of flat metrics on 
$\mathbb R ^n$, with $<\cdot, \cdot >$  being the standard scalar product on 
$\mathbb R ^n$. It is worth recalling that the family of flat metrics 
$< e^{tA} \cdot, e^{tA} \cdot >$ 
on the $\mathbb R ^n$ factor
have the same Levi-Civita connection.  This implies that the geodesic flow $(s,y) \rightarrow (s+t,y)$ acting on 
$G_A \approx \mathbb R \times \mathbb R ^n$ commutes with the parallel transport along the horospheres $\{s\} \times \mathbb R ^n$.

Turning back to $\tilde M \approx \mathbb R \times \mathbb R ^n$ with its horospherical coordinates associated to $\xi = c_{\tilde v} (+\infty)$,  
where $c_{\tilde v} $ projects to a closed geodesic $c_v$ of period $l(v)$ in $M$.
We will prove that $\tilde M$ is quasi-isometric to $G_A$, for $A$ defined by
\begin{equation}\label{A}
e^{l(v)A} = D\tilde p \, \circ \left( D(\gamma \circ \tilde g _{l(v)}(\tilde v)|E^{ss}(\tilde v) \right) \circ D \tilde p ^{-1},
\end{equation}
and where $\gamma$ is the element of the fundamental group of $M$ such that $D \gamma (\tilde g _{l(\tilde v)} (\tilde v)) = \tilde v$, by proving that $h_{l(v)k} = < e^{k A} \cdot, e^{k A} \cdot >$, for all positive integer $k$.

\smallskip

The proof of this equality reduces to 
a consequence of our assumptions that
the parallel transport along the horospheres commutes with
the flow $(s,y) \rightarrow (s+t,y)$ acting on $\tilde M \approx \mathbb R \times \mathbb R ^n$. Indeed, it follows form this commutation that  
the computation of $h_{l(v)k}(l(v)k,y) (X,X)$ for any tangent vector $X$ to $\mathbb R^n$ at the point $(l(v)k,y)$ does not depend on the point $y\in \mathbb R^n$.
Thus, it will be computed at the point $(l(v)k,y_0)$, where $y_0$ is chosen so that $(0,y_0)$ are the coordinates of the point 
$x_0 \in \tilde M$ lying on the intersection of the geodesic $c_{\tilde v}$ with the horosphere $H_\xi (0) = \mathbb R^n$; the relation $h_{l(v)k}(l(v)k,y_0) (X,X) = < e^{k A} X, e^{k A} X >$
is then easily derived from the fact that the flow $(s,y) \rightarrow (s+t,y)$ is the projection by $\tilde p$ on $\tilde M$ of the geodesic flow.

\smallskip

Let us conclude  this quick description by briefly describing how the commutation of the parallel transport along the horospheres with the geodesic flow is derived.
To this end, we adapt the construction in Avila-Santamaria-Viana \cite{ASV} and Kalinin-Sadovskaya \cite{KS}, which amounts to using the geodesic flow to construct a transportation along horospheres, which is called the stable holonomy. By construction, it is invariant by the geodesic flow. It turns out that in order to make this construction work, we need the strict $1/4$-pinching curvature assumption or the relatively $1/2$-pinched sectional curvature, which in turn corresponds to the notion of a {\sl bunched} dynamical system appearing in \cite{ASV, KS}. %Finally, we are able to show that the stable holonomy coincides with the parallel transport along the horospheres with their induced flat metric. 

\smallskip

The organization of the paper is as follows. In Section \ref{horospheres}, we show that the assumption of the main theorem on one horosphere implies that it is satisfied on all of them using the properties of the stable foliation of $T^1M$ and the density of each leaf. We also describe the geometry of the Heintze groups in the same section. In Section \ref{connection}, we describe the construction of our version of the {\sl stable holonomy}, adapted from  Avila-Santamaria-Viana \cite{ASV} and Kalinin-Sadovskaya \cite{KS}. In this Section we also prove that this new transportation, if coincides with the parallel transport for  the induced metric on {\it one} horosphere, then it is also the case for {\it all} horospheres.
Finally, in Section \ref{quasi}, this new tool allows us to prove that $\tilde M$ is quasi-isometric to the hyperbolic space, and that  
the derivative of the flow on the stable manifolds has complex eigenvalues which all have the same modulus. This  concludes the proof of  Theorem \ref{eigenvalues}, and therefore of Theorem \ref{main-thm}. 
In the Appendix, we show that 
the strong $1/4$-pinching assumption (\ref{relativepinching}) implies the bunching of the stable cocycle of the geodesic flow defined in \cite{KS}. We will also prove that the stable holonomy defined is this work
on horospheres is actually conjugate to the stable holonomy defined on the strong stable leaves of the geodesic flow.

\smallskip

\ack\  The research in this paper greatly benefited from visits of the authors at: Institut Fourier, IHP, Institut de Math\'ematiques de Jussieu-Paris Rive Gauche, Princeton University and the University of Georgia. The authors express their gratitude for their hospitality. We also thank A. Sambarino for having pointed out to us the Theorem IV of \cite{Shub} and F. Ledrappier, and X. Xie for helpful exchanges.  G\'erard Besson was supported by ERC Avanced Grant 320939, Geometry and Topology of Open Manifolds (GETOM). Sa'ar Hersonsky was partially supported by a grant from the Simons Foundation (\#319163 to Sa'ar Hersonsky). The authors are grateful to the referee for several very valuable comments; in particular, for suggesting that a relative pinching assumption on the sectional curvature may imply 
the existence of the stable holonomy, leading to an improvement of our main theorem.

\bigskip 
\bigskip

\section{Geometry of Horospheres and the Heintze groups}\label{horospheres}
\label{se:horo}
In this section, we first prove Proposition~\ref{one-all} below which, among others, proves several  continuity properties of horospheres and asserts that if one of them is flat then {\it all} horospheres are flat. We then describe the main family of examples showing that the closeness assumption in Theorem \ref{main-thm} is  necessary. These examples, consisting of simply connected Lie groups endowed with negatively curved left invariant metrics, (see \cite{He1}, Theorem 3), are due to E. Heintze and are called  ``Heintze groups". At the end of this section we provide a proof of the fact that for every $\xi \in \partial \tilde M$, the Busemann function $B(\cdot, \xi)$ is smooth.

\subsection{Geometry of horospheres}
\label{sub:flathorosphere}
Let us start by recalling a few facts about the dynamical approach describing horospheres. We first note that the strong stable and unstable distributions $\tilde E^{ss}, \tilde E^{su}$ and their associated foliations  
$\tilde W^{ss}$ and $\tilde W^{su}$ are invariant under the action of the 
fundamental group of $M$, hence they all project onto their 
natural counterparts denoted by $E^{ss}, E^{su}, W^{ss}$ and $W^{su}$ in 
$TT^1 M$ and  $T^1 M$, respectively.  An important 
consequence of the closeness of $M$ is that each leaf of the strong 
stable or unstable foliations $W^{ss}$ and $W^{su}$ is dense in $T^1 M$ (see 
\cite{An}, Theorem 15). An
application of the dynamical interpretation is described in the 
proposition below and will be important in 
the sequel. 
Given a unit tangent vector $\tilde v \in T_{z}^1 \tilde M$, we will denote by
$H_{\tilde v}$ the horosphere centered at the point $c_{\tilde v}(+ \infty) \in 
\partial \tilde{M}$ and passing through the base point $z$ of $\tilde v$. 
Observe that $H_{\tilde v} = H_\xi (s)$ where $\xi = c_{\tilde v}(+ \infty)$ and $s= B_\xi(z)$. This notation will make easier the formulation of the next Proposition.
If  $x,y \in H_{\tilde v}$ are two points such that their exists a unique geodesic 
of $H_{\tilde v}$ joining $x$ and $y$, we write $P_{H_{\tilde v}} (x,y) : T_x 
H_{\tilde v} \to T_y H_{\tilde v}$ the parallel transport along the geodesic 
path between $x$ and $y$. We will denote by $d_{H_{\tilde v}}$ the distance 
on $H_{\tilde v}$. Recall that the parallel transport is measured with respect to the induced 
Riemannian metric on $H_{\tilde v}$.

\begin{Pro}\label{one-all}
Let $M$ be a closed $(n+1)$-dimensional Riemannian manifold with negative 
sectional curvature, then the following hold.
  \begin{enumerate}
    \item Let $(\tilde {v}_k )_k$ be a sequence in $T^1 \tilde M$ such that $\lim _k \tilde {v}_k =\tilde v$. Then, 
    $H_{\tilde v_{k}}$ $C^{\infty}$-converge to $H_{\tilde v}$ on compact subsets. 
    \item It is equivalent that one or every horosphere in $\tilde M$ is flat. 
    \item There exists a positive constant $\rho >0$ such that the injectivity radius of each
    horosphere is bounded below by $\rho$.
    \item Let $(\tilde v_k )_k \in T^1_{x_k} \tilde M$ such that $\lim_k \tilde v_k =\tilde v \in T^1_{x} \tilde M$ (notice that $\lim _k x_k =x$). Let 
    $X_k \in T _{x_k} H_{\tilde v_k}$ and $y_k \in H_{\tilde v_k}$ such that $\lim _k y_k =y \in H_{\tilde v}$,
    $\lim _k X_k = X \in T_x H_{\tilde v}$ and, if $d_{ H_{\tilde v}} (x, y) < \rho$ then, 
    $\lim _{k} P_{H_{\tilde v_k}} (x_k ,y_k) (X_k) = P_{H_{\tilde v}} (x,y) (X)$.
  \end{enumerate}
\end{Pro}

\begin{proof}
Let us prove the first part of the Proposition. 
Suppose that the sequence
	$(\tilde 
v_k )_k$ is converging to $\tilde v$ in $T^1 \tilde M$.
The set of unit vectors $\tilde w$ normal to $H_{\tilde v}$ such that $[c_{\tilde w}] =[c_{\tilde v}] \in \partial \tilde M$ is the strong stable leaf $\tilde W^{ss}(\tilde v)$.
Recall that the projection 
$\tilde p : T^1 \tilde M \rightarrow \tilde M$ maps the strong stable leaf 
$\tilde W^{ss}(\tilde v)$ diffeomorphically  onto 
$H_{\tilde v} = \tilde p (\tilde W^{ss}(\tilde v))$.
Similarly, for each $k$ the horosphere $H_{\tilde {v}_k}$ is 
the projection of a strong stable leaf $\tilde W^{ss}(\tilde v_k) $,  
$H_{\tilde v _k} = \tilde p (\tilde W^{ss}(\tilde v_k))$.  Let $v_k$ and $v$ denote the 
projection under  $d\pi :  T^1\tilde M\to T^1 M$ of $\tilde v_k$ and 
$\tilde v$, where $\pi : \tilde M \to M$ is the projection. 
Let us consider a chart $U \subset T^1 M$ of the strong stable foliation $W^{ss}$ containing $v$ and let
$Q = U \cap W^{ss} (v)$ be the plaque of the foliation $W^{ss}$ through $v$. Since $U$ is a chart of the foliation $W^{ss}$,
for $k$ large enough, $U\cap W^{ss} (v_k) \neq \emptyset $ and 
the plaques $Q_k := U\cap W^{ss} (v_k)$ Hausdorff converge to $Q$.
Consequently, for the lift $\tilde Q \subset T^1 \tilde M$ of $Q$ containing $\tilde v$, the set $\tilde p (\tilde Q) \subset H_{\tilde v}$ 
is the Hausdorff limit
of the sequence of sets $\tilde p (\tilde Q_k) \subset H_{\tilde v_k}$ where $\tilde Q_k$ are lifts of $Q_k$ containing $\tilde v _k$. We will show that for all $r\geq 0$, $\tilde p (\tilde Q)$ is the limit in the $C^r$-topology, $r\geq 0$, of  
$\tilde p (\tilde Q_k)$, which will conclude the first 
part of the Proposition.

Let us choose a chart $U$ small enough so that $\tilde Q _k$ and $\tilde Q$ project diffeomorphically onto $Q_k$ and $Q$. Similarly, we can assume that
the projection $p : T^1 M \rightarrow M$ also maps diffeomorphically $Q_k$ and $Q$ into $M$. Finally, if $U$ is small enough, we have that $p(Q_k)$ and $p(Q)$ are isometrically covered by $\tilde p (\tilde Q _k)$ and $\tilde p (\tilde Q)$, respectively. We can therefore work equivalently with $p(Q_k)$ and $p(Q)$ instead of $\tilde p (\tilde Q _k)$ and $\tilde p (\tilde Q)$.
Note that for any $t_0 >0$, the strong stable foliation $W^{ss}$ of the geodesic flow $g_t$
coincide with the strong stable foliation of the diffeomorphism $g_{t_0}$, which we will 
denote by $f$. The time $t_0$ which will be fixed later on.

\smallskip

We will now apply Theorem IV.1, appendix IV, page 79, in \cite{Shub} to the diffeomorphism 
$f $ of $T^1 M$, the decomposition of $TT^1 M = E_1 \oplus E_2$ with $E_1 := \mathbb R X \oplus E^{su}$ and $E_2 := E^{ss}$. Moreover, since the geodesic flow on $T^1M$ is an Anosov flow, we can choose $t_0$ so that the following hold: 
\begin{equation}\label{H1}
\| Df(v) \| \leq \lambda \|v \|
\end{equation}
for every $v\in E_2 \backslash \{0 \}$
and 
\begin{equation}\label{H2}
\| Df(v) \| \geq \mu \|v \|
\end{equation}
for every $v\in E_1 \backslash \{0 \}$,
with the parameters $\mu =1$ and $\lambda = e^{-1}$. Notice that in (\ref{H1}) and (\ref{H2}), the norm is the Riemannian metric on $T^1 M$. The theorem mentioned above,  can now be applied while asserting  that the set of plaques $Q$ of the leaves of the strong stable foliation $W^{ss}$ of $f$, is locally a continuous family of $C^r$-embeddings into $T^1 M$, for any $r\geq 0$, of 
the unit disk $D^n$ in $\mathbb R^n$. More precisely, for $\varepsilon >0$, let us define
\begin{equation}
W^{ss}_\epsilon(v)=\Big \{ u \in T^1 M |\  d(f^n(v), f^n(u))\leq \epsilon ,\; \forall n\geq 0\,,\;\mathrm{and}\;\, d(f^n(v), f^n(u))\underset{n\to +\infty}{\longrightarrow 0}\Big\}.
\end{equation}

Let ${\mathcal E}^r(D^n, T^1M)$ denote the space of $C^r$ embeddings of $D^n$ into $T^1M$, endowed with the $C^r$ topology, where
$D^n$ is the unit disk in $\textbf{R}^n$.
Since $f$ is $C^r$, for any $r\geq 0$ the assertions of the theorem are that for every $v\in T^1 M$ we can choose a neighborhood $V$ of $v$ such that there exists a continuous map
\begin{equation}\label{theta}
\Theta : V\to {\mathcal E}^r(D^n, T^1M)\,,
\end{equation}
such that $\Theta (w) (0)=w$ and $\Theta (w)(D^n)=W^{ss}_\epsilon(w)$, for all $w\in V$. We deduce that the sequence of maps $\Theta (v_k) : D^n \to W^{ss}_\epsilon(v_k)$ converges to the map 
$\Theta (v) : D^n \to W^{ss}_\epsilon(v)$. 
We may also choose $V \subset U$ and $\epsilon >0$ small enough 
so that $p$ maps $W^{ss}_\epsilon(v_k)$ diffeomorphically into $Q_k$ for $k$ large enough and similarly,
$p$ maps  $W^{ss}_\epsilon(v)$ diffeomorphically into $Q$. 
We may also assume that $Q _k$ and $Q$ lift diffeomorphically to 
$\tilde Q _k \subset T^1 \tilde M$ and $\tilde Q \subset T^1 \tilde M$. 
We then deduce that the sequence of diffeomorphism
\begin{equation}\label{diffeo-tilde-horo}
 \alpha _k := \, \pi ^{-1} \circ p \circ \Theta (v_k) : D^n \to \tilde p (\tilde Q_k )
\end{equation}
converges to the diffeomorphism
\begin{equation}\label{diffeo-tilde-horo-lim}
 \alpha :=  \, \pi ^{-1} \circ p\circ \Theta (v) : D^n \to \tilde p(\tilde Q).
\end{equation}
which proves the first part of the Proposition.
\begin{Rem}\label{uniform-convergence}
Notice that in the above convergence, $\tilde p (\tilde Q_k ) \subset H_{\xi_{\tilde v_k}} $ 
and $\tilde p(\tilde Q) \subset H_{\xi_{\tilde v}}$ contains balls of radius $\epsilon ' := \epsilon ' (\epsilon) >0$
centered at $\tilde p( \tilde v_k )$ and $\tilde p( \tilde v )$ respectively. The above convergence therefore 
holds on open sets of uniform size.
\end{Rem}

We now prove the second part of the Proposition. Let us assume that $H_{\tilde v}$ is flat for the induced metric and consider $H_{\tilde w}$. 
Since $M$ is a closed manifold, each leaf of the strong stable foliation $W^{ss}$, in particular $W^{ss}(v)$, is dense in $T^1M$ (see \cite[Theorem 15]{An}). 
Therefore, each plaque $Q$ of $W^{ss}(w)$ contained in a chart $U \subset T^1 M$
of the foliation is the Hausdorff limit of a sequence of plaques $Q_l$ of $W^{ss}(v)$ in the same chart. Consequently, for the lift $\tilde Q \subset T^1 \tilde M$ containing $\tilde w$, the set $\tilde p (\tilde Q) \subset H_{\tilde w}$ 
is the Hausdorff limit 
of a sequence of sets $\tilde p (\tilde Q_l) \subset H_{\tilde v}$ where $\tilde Q_l$ are lifts of $Q_l$.  

Let $\Psi$ be any transversal to $W^{ss}$ passing through $w$ (for example $\Psi$ could be a neighbourhood of $w$ in its weak unstable manifold), and let $v_l$ be the intersection of $\Psi$ with the plaque $Q_l\subset W^{ss}(v)$ which approximate $Q$, that is $v_l\to w$ when $l\to +\infty$. Applying the first part of the proposition, the sequence $H_{\tilde {v}_l}$ locally converges in the $C^r$-topology to $H_{\tilde w}$. To be more precise, the metric 
$$(\pi ^{-1} \circ p\circ \Theta (w))^*(g)$$ is the pulled back to $D^n$ of the metric induced by the metric $g$ of 
$\tilde M$ on 
$\pi ^{-1}(p(\Theta (w)(D^n))) \subset H_{\tilde w}$ and, by the first part of the proposition, we deduce that 
$$(\pi^{-1}\circ p\circ \Theta (w))^*(g) = \lim _{l\to \infty} (\pi ^{-1}\circ p\circ \Theta (v_l))^*(g)$$ in the $C^{r-1}$-topology for every $r$. 
By tensoriality, the curvature of $(p\circ \Theta (w))^*(g)$ is the pulled back of the intrinsic curvature of this projected horosphere (note that the curvature depends only on the differential of $p\circ\Theta$). Since all of these quantities depend continuously on $w$, it follows that  $\tilde p(\tilde Q)$ with the induced metric is flat, just as the $\tilde p(\tilde Q_l)$ are for all $l$. 

This concludes the second part of the Proposition.

The fourth part of the proposition follows along the same lines as above. Let 
$\tilde v_k\in T^1_{x_k} \tilde M$ and $\tilde v \in T^1 _x \tilde M$ as in the statement. As above, we have convergence 
$$(\pi ^{-1} \circ p\circ \Theta (v))^*(g) = \lim _{k\to \infty} (\pi ^{-1} \circ p\circ \Theta (v_k))^*(g)$$ in the $C^{r-1}$-topology for every $r$ and
therefore the Levi-Civita connection of $(\pi ^{-1} \circ p\circ \Theta (v_k))^*(g)$ converges to the Levi-Civita of 
$(\pi ^{-1} \circ p\circ \Theta (v))^*(g)$. 
In particular, for $k$ large enough and $d_{H_{\tilde v_k}}(x_k, y_k) < \rho$ the unique geodesic between $x_k$ and $y_k$
converges to the unique geodesic joining $x$ and $y$ and thus the corresponding parallel transport along these geodesics converges. 
This concludes the proof of the fourth part of the Proposition. 

Let us prove the third part of the Proposition. We argue by contradiction assuming that
there exists a sequence $\tilde v_k \in T^1_{x_k} \tilde M$ such that the injectivity radius $\textrm{inj}_{H_{\tilde v_k}} (x_k)$ of $H_{\tilde v_k}$ at $x_k$ tends
to zero. By compactness of $M$, we may assume, after translation by elements of 
$\pi _1 (M)$, that $\tilde v_k$ converges to $\tilde v\in T^1_x \tilde M$. As above, we have convergence 
of the metrics $(\pi ^{-1} \circ p\circ \Theta (v))^*(g) = \lim _{l\to \infty} (\pi ^{-1}\circ p\circ \Theta (v_k))^*(g)$ in the $C^{r}$-topology for every $r\geq 2$, hence
the injectivity radii $\textrm{inj}_{H_{\tilde v_k}} (x_k)$ of $H_{\tilde v_k}$ at $x_k$ converges to the injectivity radius
 $\textrm{inj}_{H_{\tilde v}} (x)$ of $H_{\tilde v}$ at $x$.
Since $\textrm{inj}_{H_{\tilde v}} (x) >0$, we get a contradiction, which concludes the proof of the third part of the Proposition.

\end{proof}

\subsection{Heintze groups}\label{Heintze}

We now describe a family of  examples illustrating that the compactness of $M$ is a necessary assumption in Theorem \ref{main-thm}.
A Heintze group is a solvable group  $G_A = \mathbb R \ltimes_A \mathbb R^n$ where $A$ is an $n\times n$ matrix whose entries are real numbers. 
Such a group $G_A$ is diffeomorphic to $\mathbb R \times \mathbb R^n$ with a group action given by
$(s,y).(s',y') = (s+s', y+ e^{sA}y')$. In the sequel, we will use the coordinates given by
the diffeomorphism $\psi : \mathbb R \times \mathbb R ^n \rightarrow G_A$ defined by $\psi (s,y) := (s, e^{sA} y)$.
When the real parts of the eigenvalues of $A$ have  the same sign, E. Heintze showed the existence of left invariant metrics on $G_A$ with negative sectional curvature, see \cite{He1}.  
When the matrix $A$ is a multiple of the Identity, $G_A$ endowed with any left invariant metric is homothetic to the hyperbolic space. Furthermore, a Heintze group $G_A$  contains no cocompact lattice unless it is homothetic to the hyperbolic space, \cite{He}.

As an example, consider the $n \times n$ matrix $A$ defined by
\begin{equation}
\label{eq:exHein}
A = 
  \begin{matrix}\begin{pmatrix} a_1 & 0 &\dots &0 \\ 0 & a_2 &\dots &0 \\ \dots &\dots &\dots & \dots \\ 0 & 0 &\dots & a_n
  \end{pmatrix}\\\mbox{}\end{matrix}
\end{equation}
where  $a_1 \leq a_2 \dots \leq a_n <0$.
The left invariant metric $g$ given at $(0,0)$ by the standard
Euclidean scalar product  $dt^2 + |dy_1|^2 + \dots + |dy_n| ^2$ 
is written in the above coordinates $G_A = \mathbb R \times \mathbb R ^n$ as
\begin{equation}\label{modele}
 g=ds^2 + e^{2 a_1 s} |dy_1|^2 + \dots +e^{2 a_n s} |dy_n|^2
\end{equation}
and gives $G_A$ the structure of a Cartan-Hadamard manifold with pinched negative sectional curvature satisfying $-a_n ^2 \leq K \leq -a_1^2$. In the above coordinates and for this metric, for every $y\in \mathbb R ^n$, the curves $t \rightarrow (t,y)$ are geodesics, all being asymptotic 
to a point $\xi \in \partial G_A$ when $t\rightarrow +\infty$. For each $s\in \mathbb R$, the sets 
$\{ (s,y)\, , y\in \mathbb R^n \}$ are
horospheres $H_\xi (s)$ centered at $\xi$. For each $s$, 
the horospheres $H_\xi(s)$ are clearly isometric to the Euclidean space $\mathbb R^n$. However,  $G_A$ is isometric to the real hyperbolic space if and only if $a_1 = a_2 =\dots = a_n$ and it does not admit a compact quotient unless the $a_i$'s coincide, as proved in \cite{He}.
This exemplifies that having a family of Euclidean horospheres $H_\xi (t)$ centered at a given boundary point does not characterize the real hyperbolic space. 

Also note that the flow $\varphi _t$ defined in the above coordinates of $G_A$ by 
$$\varphi _t (s,y) := (s+t,y)$$ permutes the horospheres,
mapping $H_\xi (s)$ on $H_\xi (s+t)$.
Writing $h_s$ the metric induced by $g$ on $H_\xi (s)$ we have
$$
h_s := e^{2 a_1 s} |dy_1|^2 + \dots +e^{2 a_n s} |dy_n|^2
$$
and 
$$
\varphi _t ^{*} \,(h_{s+t}) = e^{2 a_1 (s+t)} |dy_1|^2 + \dots +e^{2 a_n (s+t)} |dy_n|^2,
$$ 
hence the two metrics $h_s$ and $\varphi _t ^{*} \,(h_{s+t})$ are linearly equivalent and therefore they share
the same Levi-Civita connexion. The flow $\varphi _t$ then preserves the Levi-Civita connexions and 
thus commutes with the parallel transport of the induced metrics on the $H_\xi (s)$'s.

\subsection{Busemann function}\label{buseman-smooth} 
Let $\tilde M$ be a Cartan Hadamard manifolds  endowed with pinched negative sectional curvature $-a^2 \leq K \leq -b^2 <0$.  The Busemann functions $B(\cdot, \xi)$
are $C^2$ for every $\xi \in \partial \tilde M$, \cite[Proposition 3.1]{Im-Hof}, and it is also known that they are $C^\infty$ in the case that $\tilde M$ is the universal cover of a closed manifold.

\smallskip

For the sake of completeness, let us give here the proof of this fact. The geodesic flow $\tilde g _t$ on $\tilde M$ is generated by the smooth vector field $Z := \frac{d}{dt}_{|t=0}\tilde g _t$
on $T^1 \tilde M$. For every $\xi \in \partial \tilde M$, the set defined by
\begin{equation}
 \tilde W^s (\xi) = \left\{\tilde v \, | \, c_{\tilde v} (+\infty) = \xi \right\} 
\end{equation}
is a weak stable leaf of $\tilde g _t$, preserved by $\tilde g_t$. It is a smooth 
submanifold of $T^1 \tilde M$ (\cite[Theorem IV.1]{Shub}) and the projection $\tilde p$ induces a diffeomorphism between $\tilde W_\xi$ and $\tilde M$. For every $\tilde v \in T^1 \tilde M$, the vector $Z(\tilde v):=\frac{d}{dt}_{|t=0}(\tilde g_t(\tilde v))$ is tangent to the flow direction at $\tilde v$ and the following holds.
\begin{equation}
D\tilde p (\tilde v)(Z(\tilde v)) = \dot{c}_{\tilde v} (0) =-\nabla B(\tilde p(\tilde v ), \xi).
\end{equation} 
Therefore, if we defined $\tilde p^{-1}(x)=\tilde v\in \tilde W_\xi$, we get $\nabla B(x,\xi) = - D\tilde p (\tilde p ^{-1} (x)) (Z(\tilde p ^{-1} (x))$
is 
a smooth vector field on $\tilde M$ and therefore $B(\cdot, \xi)$ is smooth. 

This fact will be useful in section $3$, while constructing a quasi-isometry between
$\tilde M$ and $G_A$ using horospherical coordinates.

%\newpage

\section{Stable holonomies for horospheres in negatively curved manifolds}
\label{connection}
A priori the parallel transport associated to the induced metrics on horospheres does not commute with the action of the geodesic flow. In a sharp contrast, at the end of Subsection \ref{Heintze}, we noticed that for Heintze groups it does. In this section, we  will describe another transport along horospheres, called {\sl the stable holonomy}, which by construction, commutes with the geodesic flow. A consequence of the equality of these a priori unrelated
two parallel transports is that the Levi-Civita connexions of the horospheres are flat and commute with the geodesic flow.
We will see in section $3$ that when these two properties 
hold true on the family of horospheres $H_\xi (s) \, , \, s \in \mathbb R$, for $\xi \in \partial \tilde M$
fixed by some element $\gamma \in \pi _1 (M)$, then 
$\tilde M$ is quasi-isometric to the Heintze group $G_A$, where $A$ is the derivative of the Poincar\'e first return map along
the periodic geodesic associated to $\gamma$.
\medskip

We now describe the construction of the stable holonomy following \cite{ASV} and \cite{KS}. It utilizes in a crucial way either the strong $1/4$-pinching  
or the relative $1/2$-pinching assumption on the curvature which corresponds to the 
'fiber bunched' condition of \cite {KS}, (see the Appendix). In fact, Proposition \ref{parallel-transport} and Proposition
\ref{stableholonomy-continuous} below are a consequence of Proposition 4.2 of \cite{KS}.  However, we will construct the stable holonomy in a way which is adjusted to our particular geometric setting and in order to make the paper self contained.
We conclude this section with Proposition \ref{local-global} and Corollary \ref{Cor}, stating that equality of the two transports on
a single horosphere implies equality on all horospheres. 
\smallskip

Throughout this section we will 
work with the tangent bundle of horospheres in $\tilde M$ which in turn, as level set of Busemann functions, are smooth submanifolds of the universal cover of $M$. Keeping the notations from the introduction, let $g_t :T^1 M\rightarrow T^1 M$  denotes the geodesic flow on $M$, {\sl i.e.},  the  
projection of $\tilde g_t$ under the map $D \pi :  T^1 \tilde M \rightarrow T^1 M$.
Let us choose a point $\xi\in \partial \tilde M$. It is a well known feature of the negative curvature of $\tilde M$, that any point in $\tilde M$ lies on a unique geodesic ray ending at $\xi$. Hence,  the canonical projection $\tilde p : T^1 \tilde M \rightarrow \tilde M$
induces a diffeomorphism from the set of unit vectors that are pointing in the direction of $\xi$ and $\tilde M$. This subset of unit tangent vectors will be denoted by $ \tilde W^{s}(\xi)$, and is usually called the (weak) stable manifold and the induced diffeomorphism  will be denoted by  $\tilde p_{\xi}$. 

\smallskip
With this identification, for every $t\in \mathbb R$ and for every $\xi \in \partial \tilde M$,  the action of the geodesic flow on $\tilde W^{s}(\xi)$
provides us with a one parameter group of diffeomorphism of $\tilde M$, 
\begin{equation}
\label{def-phi_t}
\varphi _{t,\xi} = \tilde p_{\xi} \circ \tilde g _t \circ {\tilde p}_{\xi}^{-1}.
\end{equation}

For $\tilde v_0\in T^{1} \tilde M$, let  $\xi = c_{\tilde v_0}(+\infty)$ and assume that  $\tilde p_{\xi}(\tilde v_0) = x_0$ with $c_{\tilde v_0}(0)=x_0$. By definition,  $\tilde p _\xi$ maps $\tilde W^{ss}(\tilde v_0)$ diffeomorphically onto the unique horosphere centered at $\xi$ which contains $x_0$. If we  denote this horosphere by  $H_\xi (0)$, then it also follows from the definitions that the derivative $D\tilde p_\xi (\tilde v_0)$ maps  $\tilde E^{ss}(\tilde v_0)$  isomorphically onto $T_{x_{0}} H_\xi (0)$. Finally, we note that the family of horospheres centered at $\xi$ can be parametrized by the time
parameter, {\sl i.e.},  for $s\in \mathbb{R}$ the horosphere $H_\xi (s)$ will denote  the unique horosphere in $\tilde M$, centered at $\xi$, which intersects the geodesic $c_{\tilde v_0}$ at time $s$.  By the property of invariance of the strong stable foliation by the geodesic flow, it follows that the diffeomorphisms $\varphi _{t,\xi}$ permutes the set of horospheres centred at $\xi$,
namely, $\varphi _{t,\xi} H_\xi (s) = H_\xi (s+t)$.  

\smallskip
We now turn to the main construction of this section, see \cite{ASV} and \cite{KS}.
The stable holonomy which we describe below provides a geodesic flow invariant way
to identify tangent spaces at different points on any fixed horosphere. 
We fix $x_0 \in \tilde M$ and recall that the horospheres are defined by
$$H_\xi (s) = \{x \in \tilde M \, | \, B_\xi(x) = s\},$$
where the Buseman function $B_\xi$ has been normalized such that $B_\xi(x_0) = 0$.

\begin{figure}[h]\label{fig-two2}
\centering
\includegraphics[width=3.75in]{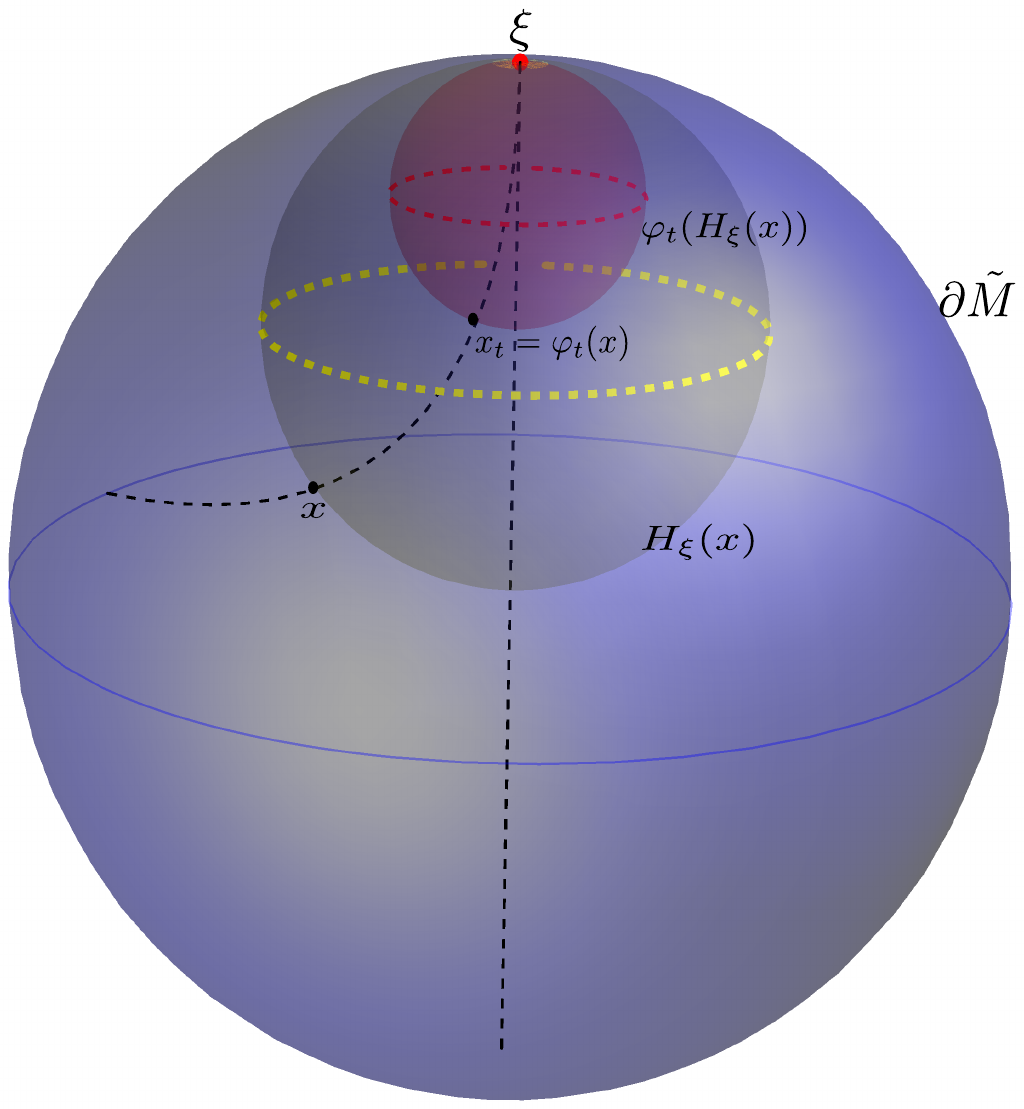}
\vspace*{-32mm}
\caption{Horospheres and action of $\varphi_t=\varphi_{t,\xi}$}
\end{figure}

We start with the following definition (see \cite[Definition 4.1]{KS}).

\begin{Def}[Stable Holonomy for Horospheres]
\label{def-stable-holonomy2} 
 
A stable holonomy is a family of maps $(x,y, \xi) \rightarrow \Pi^{\xi} _s(x,y)$, $s\in \mathbb R$, defined on the set of points $(x,y,\xi)$ such that $x,y$ belong to
the horosphere $H_\xi (s)$, and such that the following properties hold:
\begin{enumerate}
\item $\Pi ^\xi _s(x,y)$ is a linear map from $T_x H_\xi(s)$ to  $T_y H_\xi(s)$ for every $s\in \mathbb R$, $x,y \in H_\xi (s)$,
\item $\Pi ^\xi _s(x,x) = Id$ and $\Pi ^\xi _s(x,y) = \Pi ^\xi _s(z,y) \circ \Pi ^\xi _s(x,z)$ for every $s\in \mathbb R$, $x,y,z \in H_\xi (s)$
\item $\Pi ^\xi _s(x,y) = D\varphi^{-1} _{t,\xi} (\varphi _{t,\xi}(y)) \circ  \Pi ^\xi _{s+t}(\varphi_{t,\xi}(x), \varphi_{t,\xi} (y)) \circ D\varphi _{t,\xi}(x)$ for all $t \in \mathbb R$, $s\in \mathbb R,$
\item For every $\gamma \in \pi _1 (M)$, 
$
\Pi ^{\gamma \xi} _{s+B_{\gamma \xi}( \gamma x_0)} (\gamma x, \gamma y) =
D\gamma (y) \circ \Pi ^{\xi} _s (x,t) \circ (D\gamma (x))^{-1}.
$
\end{enumerate}
where in (3), $ D\varphi _{t,\xi}(z)$ denotes the differential of $\varphi _{t,\xi}$ at the point $z$.
\end{Def}

Notice that condition (2) tells that this stable holonomy, if it exists, is 'flat' and 
condition (4) that the stable holonomy is equivariant under the action 
of the fundamental group of $M$ on the set of horospheres.

Let us choose a point $\xi \in  \partial \tilde M$. In the sequel of this section, we will set $\varphi_{t}:=\varphi_{t,\xi}, t\in\mathbb{R}$
and $\tilde p_\xi = \tilde p$. Recall that the induced Riemannian metric on $H_\xi (t)$ is denoted by $h_t$, and let 
$\nabla ^t$ denote the Levi-Civita connection associated to $h_t$. 
The parallel transport with respect to $\nabla ^t$, along any path joining any two points $x$ and $y$ in $H_\xi (t)$, is an isometry between
$T_x H_\xi (t)$ and $T_y H_\xi (t)$. 
The isometry a priori depends on the path.  However, if $x, y$ in $H_\xi (t)$ are at distance less than the injectivity radius 
of $H_\xi (t)$, there exists a unique geodesic segment joining $x$ and $y$ and we will therefore denote by 
\begin{equation}\label{parallel}
P_t ^\xi (x,y)
\end{equation}
the parallel transport along this segment.

\smallskip
We now turn to the main proposition of this section that will grant us the existence of the stable holonomy along horospheres. It is a reformulation of \cite[Proposition 4.2]{KS} or of \cite[Proposition 2.2]{B}. Since we will use the construction later on, we will shortly describe it. We first need two lemmas.

\smallskip

The first  lemma gives uniform contraction properties of the 
maps $\varphi _t$ under the strong $1/4$-pinching condition on the curvature of $M$.
Let us normalize the sectional curvature $K$ of $M$, so that the following inequalities are satisfied for some constant  $1>\tau >0$,
\begin{equation}\label{pinching}
-4(1-\tau ) \leq K \leq -1.
\end{equation}
\begin{Lem}
\label{contraction-dilation}
Let $x,y$ be two points on $H_\xi (s)$ and let $X$ be a tangent vector in $T_xH_\xi (s)$. 
Then, for any $t \geq 0$, the following estimates hold  
\begin{enumerate}
\item
$\| D\varphi _t (x)( X) \|_{h_{s+t}} \leq e^{-t} \|X\|_{h_s}$,
\smallskip
\item
$\| D\varphi _t ^{-1}(x)(X)\|_{h_{s-t}}\leq e^{(2\sqrt{1-\tau})t} \| X\|_{h_s} \leq e^{(2-\tau)t} \| X\|_{h_s},$ and
\smallskip
\item
$d_{h_{s+t}}(\varphi _t (x), \varphi _t (y)) \leq e^{-t} d_{h_s}(x,y)$.
\end{enumerate}
\end{Lem}
\begin{proof}
The norm and the distance we use above is computed with respect to the  induced Riemannian metric on the corresponding horosphere. Recall that 
a  {\it stable} Jacobi field $Y(t)$ along a geodesic ray $c_{\tilde v} (t)$, $t>0$, is a bounded Jacobi field, see  \cite[Definition 2.1]{Im-Hof}. The proof of these inequalities is a direct consequence of the estimate of the growth of the  stable Jacobi fields as done  in  \cite[Theorem 2.4]{Im-Hof}. 

\smallskip

In fact, we only need to show that 
$D\varphi _t (X)$ is a stable Jacobi field. This follows from the Anosov property of the geodesic flow of $M$, see \cite[Appendice 21]{ArAv}. Indeed, if $X$ is a tangent vector in $T_x H_\xi (s)$ at the point $x$, then $X= D\tilde p (\tilde v)(V)$, where $V \in E^{ss}(\tilde v) \subset T_{\tilde v} T^1 \tilde M$, and  $\tilde v$ is the unit vector in $T_x \tilde M$
perpendicular to $H_\xi (s)$ and pointing toward $\xi$. Therefore, by applying the chain rule to Equation~(\ref{def-phi_t}), and recalling that $x={\tilde p}({\tilde v})$, we obtain that

\begin{equation}
\label{eq:flowderivative}
D\varphi _t (x)(X) = D\tilde p (\tilde g_t (\tilde v))\circ D\tilde g_t (\tilde v)(V).
\end{equation}

\smallskip

Since the geodesic flow of $M$ is Anosov and $V\in E^{ss} (\tilde v)$, it follows that  
\begin{equation}
\label{eq:norm1}
\lim _{t \to \infty} \| D\tilde g _t (\tilde v)(V) \| =0,
\end{equation}
which implies that  $ \lim _{t \to \infty} \|D\varphi _t (x)(X) \|=0$. Indeed the map $\tilde p : T^1 \tilde M \rightarrow \tilde M$ is defined on the quotient (by $\pi_1(M)$) by $p: T^1 M \rightarrow M$, the compactness of $M$ grants us that $\tilde p$  as well as $D \tilde p$ are  bounded.
Hence, it follows that $D\varphi _t (X)$ is a stable Jacobi field and this concludes the proof of the first assertion of the Lemma \ref{contraction-dilation}. The other assertions follow easily.

\end{proof}

Since $(\tilde M, \tilde g)$ covers the  closed manifold $(M, g)$,  for each $\sigma \in [0,1]$, we are able to obtain a uniform control on the action of $\varphi_\sigma$ as follows. 
We first study the behavior of the family of horospheres $H_\xi (s)$, $s\in \mathbb R$, orthogonal to 
the geodesic $c_{\tilde v}(s)$ such that $c_{\tilde v}(+\infty) = \xi$. 
By assertion (3) of  Proposition \ref{one-all}, we will assume from now on that the injectivity radius of every horosphere is bounded below
by $\rho >0$. For each $x\in H_\xi (s)$, we denote $c_x$ the geodesic passing through $x$ asymptotic 
to $\xi$, ie. $c_x (+\infty) = c_{\tilde v} (+\infty) =\xi$ parametrized in such a way that $c_x (s) =x$.

\begin{Lem}
\label{square-estimate}
For every $R>0$, there exists a constant $C_R >0$ such that for any $s\in \mathbb R$, any $\sigma \in [0,R]$, any
two points $x,y \in H_\xi (s)$ such that  $d_{H_\xi (s)} (x,y) < \rho$, and any $X \in T_x H_\xi (s)$, the following holds.  
\begin{equation}
\label{eq:sigma}
\| \left(D\varphi _\sigma ^{-1} (\varphi_\sigma (y))\circ P^\xi _{s+\sigma} (\varphi _\sigma (x), \varphi _\sigma (y)) \circ D\varphi _\sigma (x)  - P^\xi _s (x,y)\right) (X) \|_{h_s}\leq C_R d_{h_s}(x,y)\,\|X\|_{h_s}.
\end{equation}
\end{Lem}
\begin{proof}
Let us first assume that $X\in T_xH_\xi(s)$ has a unit norm. 
Define $X_\sigma := D\varphi _\sigma (x) X$ and let
$c: [0,d] \to H_\xi (s)$ be the geodesic segment of $H_\xi (s)$ between $x$ and $y$, where $d=d_{h_{s}}(x, y)$. 
Let  
$
c_\sigma(u) : [0,d]\longrightarrow H_\xi (s+\sigma ),
$
be the geodesic segment, parametrized with constant speed, joining $\varphi_\sigma (x)$ and $\varphi_\sigma (y)$ which exists by Lemma \ref{contraction-dilation}, (3). Notice that also by Lemma \ref{contraction-dilation} we have 
\begin{equation}\label{c-sigmadot}
e^{-(2-\tau)} \leq |\dot{c}_{\sigma}| \leq 1.
\end{equation}
We have,
\begin{eqnarray*}
&D\varphi _\sigma ^{-1} (\varphi_\sigma (y))\circ P^\xi _{s+\sigma} (\varphi _\sigma (x), \varphi _\sigma (y)) \circ D\varphi _\sigma (x)  - P^\xi _s (x,y) =\\
&\int_0^d \frac{d}{du} \left(D\varphi _\sigma ^{-1} (c_\sigma (u))\circ 
\left(P^\xi _{s+\sigma} (\varphi _\sigma (x), c_\sigma (u)) \circ D\varphi _\sigma (x) 
-D\varphi _\sigma (c(u)) \circ P^\xi _s (x,c(u)) \right) \right)du.
\end{eqnarray*}

\begin{comment}
\begin{align}
X_\sigma (\varphi_\sigma (y))-P _{c_\sigma} (\varphi_\sigma (x), \varphi _\sigma (y))(X_\sigma (\varphi (x)) & = X_\sigma (c_\sigma (d))-P _{c_\sigma} (c_\sigma (0), c_\sigma (d))(X_\sigma (c_\sigma (0)))\\ \nonumber
 & = \int_0^d \frac{d}{du} \Big ( P _{c_\sigma} (c_\sigma (u), \varphi _\sigma (y))(X_\sigma (c_\sigma (u)))\Big ) du.
\end{align}
Note that in the previous formulas the computations involve only vectors belonging to $T_{\varphi_\sigma (y)}H_\xi (s+ \sigma )$.
By the definition of the parallel transport along the curve $c_\sigma$, we have  
\begin{align}
\frac{d}{du} \Big ( P _{s+\sigma}^\xi (c_\sigma (u), \varphi _\sigma (y))(X_\sigma (c_\sigma (u)))\Big ) & = P _{s+\sigma}^\xi (c_\sigma (u), \varphi _\sigma (y)) \Big ((\nabla^{s+\sigma}_{D\varphi_\sigma ({\partial\over{\partial u}})}X_\sigma )(c_\sigma(u))\Big )\,,\\ \nonumber
& = P _{s+\sigma}^\xi (c_\sigma (u), \varphi _\sigma (y)) \Big ((\nabla^{s+\sigma}_{Y_\sigma } X_\sigma )(c_\sigma(u))\Big )\, ,
\end{align}
where $\nabla^{s+\sigma}$ is the Levi-Civita connection of the metric $h_{s+\sigma }$, along the horosphere $H_\xi (s+\sigma )$, and $\partial\over{\partial u}$ is the tangent vector field $\dot c (u)$ along the curve $c$. 

\smallskip

It now remains to estimate $\nabla^{s+\sigma}_{Y_\sigma } X_\sigma $. 
\end{comment}

By compactness of $M$ and by (\ref{def-phi_t}), the norm of every covariant derivative of $\varphi ^\xi _\sigma$
and $(\varphi ^\xi _\sigma)^{-1}$,
$\xi \in \partial  \tilde M$ and $\sigma \in [0,R]$ is 
bounded above by a constant depending on the degree of derivation. In particular, there exists a constant $C_R>0$
such that the integrand in the right hand side term above is bounded above by $C_R$.

We deduce that
\begin{equation}
\label{eq:finalbound}
\| P^\xi_s (x,y) (X)- D\varphi_\sigma^{-1}(\varphi_\sigma (y))\circ P^\xi_{s+\sigma}(\varphi _\sigma (x), \varphi _\sigma (y))\circ D\varphi _\sigma (x)(X) \|_{h_{s}} \leq C d_{h_s}(x,y).
\end{equation}
If the norm of $X$ is not equal to $1$,  the desired inequality follows by simple modifications of the proof above.

\end{proof}

\begin{Rem}
Notice that the constant $C$ in the above proposition does not depend on the horosphere $H_\xi (s)$ nor even on $\xi$. More precisely, in formula \ref{eq:finalbound} the parallel transport operators are isometries, hence their norms are bounded by one. Only the differential of $\phi_\sigma$ matters. These maps, for $\sigma\in[0,1]$ are projections, by $\tilde p$ to $\tilde M$, of the geodesic flow on $T^1\tilde M$ restricted to the submanifolds $\tilde W^s(\xi )$. Now by compactness of $M$, $T^1(M)$ and $[0,1]$, $\tilde p$ and the geodesic flow on $T^1\tilde M$ have bounded derivatives at any order. Finally the arguments in Subsection \ref{sub:flathorosphere} show that the manifolds $\tilde W^s(\xi)$ have uniformly bounded geometry at any order with constants independent of $\xi$.

Notice however that independence on $\xi$ is not really needed in our argument.
\end{Rem}

We now turn to prove the existence of a stable holonomy. In the following Proposition, we assume that 
the sectional curvature satisfies 
either the strong $1/4$-pinching or relative $1/2$-pinching assumption. We will then describe possible generalisations based on the results in \cite{HB2}. However stable holonomy may exist without
any pinching assumption but just under the negativity of the sectional curvature. We don't know any counterexample to this. For every $\tilde v \in T^1 \tilde M$, we consider 
the family of horospheres centered at $\xi := c_{\tilde v} (+\infty)$, which we parametrize as $H_\xi (t)$, $t\in \mathbb R$, where the parameter $t=0$ corresponds to the horosphere containing the base point of $\tilde v$.

\begin{Pro}
\label{parallel-transport}
Let $M$ be a closed Riemannian manifold with pinched negative curvature satisfying either 
the strong $1/4$-pinching condition $-4(1-\tau ) \leq \kappa \leq -1$ or the relative $1/2$-pinching condition.
Let $\tilde v$ be a unit vector tangent to $\tilde M$.
Let $\xi = \lim_{t\to +\infty}c_{\tilde v}(t) \in  \partial \tilde M$. Then, 
\begin{enumerate}[(i)]
\item For every $s \in \mathbb R$, $x,y \in H_\xi (s)$, there exists a linear map 
$$\Pi ^\xi _s (x, y) : T_x H_\xi (s) \to T_y H_\xi (s)$$
satisfying  conditions (1), (2), (3) in Definition (\ref{def-stable-holonomy2}),
\item $\| \Pi ^\xi _s (x, y) - P_s ^\xi (x,y)\| \leq C d_{h_s}(x,y)$ for all $x,y$ such that $d_{h_s} (x,y) < \rho$.
\item  Properties (i) and (ii)  uniquely determine the stable holonomy.
\item The stable holonomy is $\pi _1 (M)$-equivariant, ie for every $\gamma \in \pi _1 (M)$, we have
$$ \Pi ^{\gamma \xi} _{s+B_{\gamma \xi}( \gamma x_0)} (\gamma x, \gamma y) =
D\gamma (y) \circ \Pi ^{\xi} _s (x,t) \circ (D\gamma (x))^{-1}.$$
\end{enumerate}
\end{Pro}
\begin{proof}
The proof follows closely the methods given in \cite[Proposition 4.2]{KS}. We reproduce here 
only the part of the construction, modified to our setting, which we will need in the sequel.
The proof is organized into the two cases corresponding to the different curvature assumptions.
%\textcolor{blue}{We then describe possible generalisations based on the results in \cite{HB2}.}

{\bf The case of strong $1/4$-pinching of the curvature.}

\smallskip

Recall that we have denoted $\varphi _t := \varphi _{\xi, t}$.
Let us consider $x,y \in H_\xi (s)$ such that $d_{H_\xi (s)} (x,y) \leq R $, for some fixed $R$. For every $x\in H_\xi (s)$, denote $x_t := \varphi _t (x) \in H_\xi (s+t)$.
By Lemma \ref{contraction-dilation} (3),
there exists $t_0 := t_0 (R) \geq 0$ such that $d_{H_\xi (s+t_0)} (x_{t_0}, y_{t_0}) < \rho$.
\begin{comment}
For $s,t \in \mathbb R$, consider the diffeomorphism $\varphi _t : H_\xi (s) \rightarrow H_\xi (s+t)$ and the pulled back 
flat metric $h_{s,t} = \varphi _t ^{*} h_{s+t}$ on $H_\xi (s)$,
that is, for $ x\in H_\xi (s)$ we have  
\begin{equation}
\label{eq:pullbackmetric}
h_{s,t} (x)(\cdot,\cdot) = h_{s+t}(D\varphi _t (x)(\cdot), D \varphi _t (x)(\cdot)).
\end{equation}
The parallel transport, $\Pi_{s,t}^\xi$, with respect to the metric $h_{s,t}$, is therefore given by   

\begin{equation}
\label{eq:pullbackpara}
\Pi _{s,t}^\xi =\varphi _t ^{*} \left(P _{s+t}^\xi \right).
\end{equation}
Precisely,  for $x,y \in H_\xi (s)$ we have  
\begin{equation}\label{paralleltransport-t}
\Pi _{s,t} ^\xi (x,y) = D\varphi _t ^{-1} (\varphi _t (y)) \circ P _{s+t}^\xi (\varphi _t (x), \varphi _t (y)) \circ  D\varphi _t(x).
\end{equation}
\end{comment}
\smallskip 
 
Let us turn to proving assertion (i). For every $t\in \mathbb R$, define 
$$
c_t : [0,1] \to H_\xi (s+t)
$$
the geodesic segment, parametrized with constant speed, between $x_t$ and $y_t$ 
which is well defined when their distance is less than $\rho$.

For $x,y \in  H_\xi (s)$ we define
\begin{equation}\label{stableholonomy=lim}
\Pi ^\xi _s (x,y) = \lim _{t\to \infty} d\varphi _{t} ^{-1} (y_{t}) \circ  P^\xi _{t} (x_{t}, y_{t}) \circ d\varphi_{t}(x) .
\end{equation} 
Note that the term $P^\xi _{t} (x_{t}, y_{t})$ in the limit is well defined for all  $t\geq t_0$
since the distance between $x_t$ and $y_t$ is decreasing. Let us show that the above limit exists. Define for $j\geq 0$, $x,y \in H_\xi (s)$,
$$
\Pi^\xi _{s,j} (x,y) := 
d\varphi _{t_0 +j} ^{-1} (y_{t_0 +j}) \circ  P ^\xi_{s+t_0 +j} (x_{t_0 +j}, y_{t_0 +j}) \circ d\varphi_{t_0 + j}(x) .
$$

We have for every $N\geq 0$,
\begin{equation}\label{approx-N}
\Pi^\xi _{s,N} (x,y) = \Pi^\xi _{s,0} (x,y) + \sum _{j=0}^{N-1} \left(\Pi^\xi _{s,j+1} (x,y) - \Pi^\xi _{s,j} (x,y)\right).
\end{equation}

Each term in the above sum is expanded as
\begin{eqnarray*}
&\left( \Pi _{s,j+1}^\xi  - \Pi _{s,j} ^\xi \right) (x,y) =  \\
& D\varphi _{t_0+j} ^{-1} (y_{t_0 +j}) \circ \big[  D\varphi _1 ^{-1}(y_{t_0 +j+1})\circ P^\xi _{s+ t_0 +j+1} (x_{t_0 +j+1},y_{t_0 +j+1})\circ 
D\varphi _1 (x_{t_0 +j}) -P_{c_j}  (x_{t_0 +j},y_{t_0 +j})\big] \circ D\varphi _{t_0+j} (x),  
\end{eqnarray*}
hence, by Lemma \ref{square-estimate}, we get
\begin{equation}
\label{norm-square}
\| \left( \Pi _{s,j+1}^\xi  - \Pi _{s,j}^\xi \right) (x,y) \| \leq 
C \|D\varphi _{t_0 +j} ^{-1} (y_{t_0 +j})\| \|D\varphi _{t_0 +j} (x) \| d_{h_{s+t_0 +j}}(x_{t_0 +j},y_{t_0 +j}).
\end{equation}

Assertion (3) of Lemma \ref{contraction-dilation} implies that 
\begin{equation}\label{con-unif}
d_{h_{t_0 +s+j}}(x_{t_0 +j}, y_{t_0+j})\leq e^{-{(t_0 +j)}}d_{h_s}(x,y).
\end{equation}
Substituting this inequality back in (\ref{norm-square}) and using the estimates
(1) and (2) of Lemma \ref{contraction-dilation} yield that 
\begin{equation}
\label{summand-estimate}
\| \left( \Pi _{s,j+1}^\xi  - \Pi _{s,j}^\xi \right) (x,y) \| \leq 
C e^{-\tau (t_0 +j)}d_{h_s}(x,y).
\end{equation}
Therefore, the limit in (\ref{stableholonomy=lim}) exists and is well defined. The $\pi _1 (M)$-invariance is 
obvious and proofs of the others parts of this proposition are the same as in those of Theorem 4.2 of \cite{KS}.

\smallskip

{\bf The case of relative $1/2$-pinching of the curvature}\newline
In this case we use a result of B. Hasselblatt stating that the geodesic flow $g_t$ of
a closed manifold with a relatively $1/2$-pinched negative curvature satisfies the following 'bunching' condition,
%which we write equivalently for the lifted flow $\tilde g _t$ on $T^1 \tilde M$
\cite{HB} Theorem 6. The geodesic flow $g _t$ on $T^1 M$ is $\alpha$-bunched, $\alpha >0$,
if there exists functions $\mu _{\pm} : T^1 M \times \mathbb R_+  \rightarrow (0,1)$ such that for every $v \in T^1 M$, $X \in E^{ss} (v)$ and $t>0$,
\begin{equation}\label{bunch1}
\mu _{-}(v, t) \|X\| \leq \| D g _t (X) \| \leq \mu _{+}( v, t) \|X\|
\end{equation}
and
\begin{equation}\label{bunch2}
\lim_{t\to \infty} \sup _{v\in T^1 M} \mu _{+}( v, t)^{2/\alpha} \mu _{-}( v, t)^{-1} =0.
\end{equation}
%In \cite {HB}, B. Hasselblatt proves
\begin{Thm}[\cite{HB}, Theorem 6]\label{theo:1/2}
Let $M$ be a closed Riemannian manifold with relative $1/2$-pinched negative curvature.
Then the geodesic flow of $M$ is $(1+\epsilon)$-bunched for some $\epsilon >0$.
\end{Thm}
It turns out that in the proof of this Theorem, it is shown that the convergence in (\ref{bunch2})
is exponential, ie. there exists $\tau >0$ and $A>0$ such that
\begin{equation}\label{bunch3}
\sup _{v\in T^1 M} \mu _{+}( v, t)^{2/\alpha} \mu _{-}( v, t) ^{-1} \leq A \,e^{-\tau t}.
\end{equation}
The proof of the Proposition in the relative pinching case \ref{parallel-transport} is then similar
to the one under the strong pinching assumption.  We first notice that (\ref{bunch1}) and 
(\ref{bunch2})
lift to the universal cover into
\begin{equation}\label{tildebunch1}
\tilde \mu _{-}(\tilde v, t) \|X\| \leq \| D \tilde g _t (X) \| \leq \tilde \mu _{+}( \tilde v, t) \|X\|
\end{equation}
and
\begin{equation}\label{tildebunch2}
\sup _{\tilde v\in T^1 \tilde M} \tilde \mu _{+}( \tilde v, t)^{2/\alpha} 
\tilde \mu _{-}( \tilde v, t)^{-1} \leq A\, e^{-\tau t},
\end{equation}
where $t>0$, $X\in E^{ss} (\tilde v )$ and the functions
$\tilde \mu _{\pm} : T^1 \tilde M \times \mathbb R_+  \rightarrow (0,1)$ are invariant under the action 
of the fundamental group of $M$ on the first variable.
By (\ref{def-phi_t}), recall that 
$$
\varphi _{t,\xi} = \tilde p_{\xi} \circ \tilde g _t \circ {\tilde p}_{\xi}^{-1}.
$$
For $\tilde v \in T^1 \tilde M$, $x= \pi (\tilde v)$ and $\xi = c_{\tilde v} (+\infty)$, denote by
$$
\mu _{\pm} ^{\xi} (x,t) := \tilde \mu _{\pm} (\tilde v, t).
$$
Since there exists $C>0$ such that $C^{-1} \leq \| D\tilde p_{\xi}^{\pm 1} \| \leq C$ by compactness of $M$,
the above relations (\ref{tildebunch1}) and (\ref{tildebunch2}) translates into
\begin{equation}\label{phi-bunch1}
C ^{-2} \mu _{-}^\xi (x, t) \|X\| \leq \| D \varphi _{t,\xi} (X) \| \leq C^2 
\mu _{+}^\xi ( x, t) \|X\|
\end{equation}
and
\begin{equation}\label{phi-bunch2}
\sup _{x\in \tilde M, \xi \in \partial \tilde M}  \mu _{+} ^\xi (x, t)^{2/\alpha} 
\mu _{-}^\xi (x, t) ^{-1} \leq A \,e^{-\tau t},
\end{equation}
where $X$ is a vector tangent at $x$ to the horosphere
centered at $\xi$ and passing through $x$.
Since $\alpha = 1+\epsilon$ and 
$\mu_{+}^{\xi} (x,t) = \tilde{\mu} _{+} (\tilde v, t) <1$ we can choose 
by (\ref{phi-bunch2}) $t_0>0$ and $0<\theta <1$ such that for every 
$x\in \tilde M$,
\begin{equation}\label{t_0}
\mu _{+} ^\xi (x, t_0)^{2} \mu _{-}^\xi (x, t_0) ^{-1} \leq \theta
\end{equation}
We now argue as in the previous case. We define
$$
\Pi^\xi _{s,j} (x,y) := 
d\varphi _{t_0 (1+j)} ^{-1} (y_{t_0 (1+j)}) \circ  P ^\xi_{s+t_0(1 +j)} (x_{t_0(1 +j)}, y_{t_0 (1+j)}) \circ d\varphi_{t_0 (1+ j)}(x) 
$$
and get similarly as in (\ref{norm-square})
\begin{equation}
\label{phi-norm-square}
\| \left( \Pi _{s,j+1}^\xi  - \Pi _{s,j}^\xi \right) (x,y) \| \leq 
C_1 \|\left(D\varphi _{t_0 (1 +j)} (y)\right)^{-1}\| \|D\varphi _{t_0(1 +j)} (x) \| d_{h_{s+t_0 (1 +j)}}(x_{t_0(1 +j)},y_{t_0 (1+j)}).
\end{equation}
There is here a slight difference with the previous case coming from the fact that 
the term $\|\left(D\varphi _{t_0 (1 +j)} (y)\right)^{-1}\| \|D\varphi _{t_0(1 +j)} (x) \|$ in 
(\ref{phi-norm-square}) cannot be estimated through 
the estimates of Lemma \ref{contraction-dilation} which are uniform in $x$ and $y$.
Instead, we argue like in Lemma 4.3 of \cite{KS} : denoting $x_k := \varphi _{t_0 (1 +k)} (x)$,
and similarly replacing $x$ by $y$,
we have
$$
D\varphi _{t_0 (1 +j)} (x)= D\varphi _{t_0} (x_{t_0 (1 +j-1)}) \circ .... \circ D\varphi _{t_0} (x),
$$
and
$$
D\varphi _{t_0 (1 +j)} (y)= D\varphi _{t_0} (y_{t_0 (1 +j-1)}) \circ .... \circ D\varphi _{t_0} (y),
$$
hence
\begin{equation}\label{product1}
\|\left(D\varphi _{t_0 (1 +j)} (y)\right)^{-1}\| \|D\varphi _{t_0(1 +j)} (x) \| 
\leq \Pi _{k=0}^{k=j-1} \left( \| D\varphi _{t_0}(y_k)^{-1} \| \| D\varphi _{t_0}(x_k) \|\right),
%\Pi _{k=0}^{k=j-1}\frac{\| D\varphi _{t_0}(x_k)\|}{\| D\varphi _{t_0}(y_k) \|}.
\end{equation}
therefore,
\begin{equation}\label{product2}
\|\left(D\varphi _{t_0 (1 +j)} (y)\right)^{-1}\| \|D\varphi _{t_0(1 +j)} (x) \| 
\leq \Pi _{k=0}^{k=j-1} \left( \| D\varphi _{t_0}(y_k)^{-1} \| \| D\varphi _{t_0}(y_k) \|\right)
\Pi _{k=0}^{k=j-1}\frac{\| D\varphi _{t_0}(x_k)\|}{\| D\varphi _{t_0}(y_k) \|}.
\end{equation}
Let us estimate the last product in (\ref{product2}). 
Since for every $x\in \tilde M$ we have 
$C_2 ^{-1} \leq \|D\varphi_{t_0}(x)\| \leq C_2$ for 
$C_2 = C^2 \sup _x (\mu _{+}^{\xi} (x,t_0))$ by (\ref{phi-bunch1}),
we deduce from Lemma \ref{square-estimate} that 
\begin{equation}\label{eq:sigma1}
\| \left(P^{\xi} _{s+t_{0}} (\varphi _{t_0} (x), \varphi _{t_0} (y)) \circ D\varphi _{t_0} (x)  
- D\varphi _{t_0} (y) \circ P^\xi _s (x,y)\right) (X) \|_{h_s}\leq C_3\,C_R d_{h_s}(x,y)\,\|X\|_{h_s},
\end{equation}
and since the parallel transport is an isometry we have
\begin{equation}\label{eq:sigma2}
| \|D\varphi _{t_0} (x)\| - \|D\varphi _{t_0} (y)\| | 
%\| \left(P^{\xi} _{s+t_{0}} (\varphi _{t_0} (x), \varphi _{t_0} (y)) \circ D\varphi _{t_0} (x)  
%- D\varphi _{t_0} (y) \circ P^\xi _s (x,y)\right) (X) \|_{h_s}
\leq C_3\,C_R d_{h_s}(x,y).
\end{equation}
We therefore get
\begin{equation}\label{eq:sigma3}
\vert \left(1-\frac{\|D\varphi _{t_0} (x)\|}{\|D\varphi _{t_0} (y)\|} \right)\vert
%\| \left(P^{\xi} _{s+t_{0}} (\varphi _{t_0} (x), \varphi _{t_0} (y)) \circ D\varphi _{t_0} (x)  
%- D\varphi _{t_0} (y) \circ P^\xi _s (x,y)\right) (X) \|_{h_s}
\leq C_4\,C_R d_{h_s}(x,y),
\end{equation}
where $C_4 = C_2 \, C_3$.
Now, recalling that the sectional curvature of $M$ satisfy $K \leq -a^2 <0$ for some $a$,
we deduce that $d(x_k, y_k) \leq e^{-at_0 (1+k)}$ and thus that there exists $C_5 >0$ such that for 
every $k$,
\begin{equation}\label{eq:sigma4}
\Pi _{k=0}^{k=j-1}\frac{\| D\varphi _{t_0}(x_k)\|}{\| D\varphi _{t_0}(y_k) \|}
\leq C_5.
\end{equation}
From (\ref{phi-norm-square}), (\ref{product2}) and (\ref{eq:sigma4}), we obtain
\begin{equation}
\label{phi-norm-square1}
\| \left( \Pi _{s,j+1}^\xi  - \Pi _{s,j}^\xi \right) (x,y) \| \leq 
C_6 \Pi _{k=0}^{k=j-1} \left( \| D\varphi _{t_0}(y_k)^{-1} \| \| D\varphi _{t_0}(y_k) \|\right)
d_{h_{s+t_0 (1 +j)}}(x_{t_0(1 +j)},y_{t_0 (1+j)}).
\end{equation}
From (\ref{phi-bunch1}) and (\ref{phi-norm-square1}) we get
\begin{equation}
\label{phi-norm-square2}
\| \left( \Pi _{s,j+1}^\xi  - \Pi _{s,j}^\xi \right) (x,y) \| \leq 
C_6 \Pi _{k=0}^{k=j-1} \left( \left(\mu ^{\xi} _{-} (y_k, t_0)\right)^{-1}
 \mu ^{\xi} _{+} (y_k, t_0)\right)d_{h_{s+t_0 (1 +j)}}(x_{t_0(1 +j)},y_{t_0 (1+j)}).
\end{equation}
Now, similarly as in Lemma 1.1 of \cite{BW} we see that
\begin{equation}\label{distance}
d_{h_{s+t_0 (1 +j)}}(x_{t_0(1 +j)},y_{t_0 (1+j)}) \leq C_7 \,\Pi _{k=0}^{k=j-1}\mu ^{\xi} _{+} (y_k, t_0),
\end{equation}
hence from (\ref{phi-norm-square2}) we deduce 
\begin{equation}
\label{phi-norm-square3}
\| \left( \Pi _{s,j+1}^\xi  - \Pi _{s,j}^\xi \right) (x,y) \| \leq 
C_8 \Pi _{k=0}^{k=j-1} \left( \left(\mu ^{\xi} _{-} (y_k, t_0)\right)^{-1}
 \left(\mu ^{\xi} _{+} (y_k, t_0)\right)^2\right).
\end{equation}
By (\ref{t_0}) we therefore have
\begin{equation} 
\label{phi-norm-square3}
\| \left( \Pi _{s,j+1}^\xi  - \Pi _{s,j}^\xi \right) (x,y) \| \leq C_9 \,\theta ^{j+1}
\end{equation}
and we conclude as in the previous case.

\end{proof}
{\bf General relative pinching.}
Theorem \ref{theo:1/2} can be extended; indeed, following \cite{HB2}, one could consider the situation of a more general relative pinching. In order to get a result with the same technique than in \cite{HB}, that is comparison theorems for the Riccati equations we need to combined strong pinching and relative pinching.
More precisely, let us recall the statement of Theorem 4.3 in \cite{HB2}.

\begin{Thm}\label{theo:general}
Let $a$ and $b$ satisfy $0\leq b \leq a\leq 1$. The geodesic flow of a closed negatively curved Riemannian manifold which is $b$-pinched and relatively $a$-pinched  is 
$$C(a, b)+\epsilon-\textrm{bunched}\,,$$
for some positive $\epsilon$, where $C(a, b)=a-b+\sqrt{(a+b)^2+4(1-a)b}$.
\end{Thm}
Like in the article \cite{HB} the $\epsilon >0$ is small and it appears because how pinching assumptions are strict. Notice that when the upper sectional curvature approaches zero, even though a relative pinching is given the comparison may not give the bunching, the role played by the strong pinching is then to circumvent this difficulty. In \cite{HB2} some interesting explicit solutions and evidences of the optimality of this result are given.\\

Now, if $a$ and $b$ are chosen so that $C(a, b)=1$ we get $(1+\epsilon)$-bunching and are able to apply a proof similar to the one given for Theorem \ref{theo:1/2}. A direct computation shows that $C(a, b) = 1$ is equivalent to $a+b=1/2$. The following remark then gives some useful informations.

\begin{Rem} 
Let us first remark (see \cite{HB2}, Remark 1.5) that, under the hypotheses given in Theorem \ref{theo:general}: $C(a,b)\geq 2a$ and $C(a,a)=2\sqrt{a}$. It is furthermore obvious to check that if $a=1/2$ then $b=0$ yields $(1+\epsilon)$-bunching which means nothing more than the sectional curvature is non positive, this is our second case. Also, if $a=1/4$ then $b=1/4$ which means that we have a strong quarter-pinching and not only a relative one; this is our first case. These are the only cases for which only one pinching condition is necessary. To go further, note that the bunching condition satisfies a monotonicity property; indeed, if $\alpha \geq  \beta$, then $\alpha$-bunched implies $\beta$-bunched. A straightforward consequence  is that the condition we really need in order to make the construction of the stable holonomy is that $C(a,b) \geq 1+\epsilon$, for some small $\epsilon$, which is equivalent, by the same direct computation to $a+b >1/2$. Consequently $a = 1/4 + \eta'$ and $b = 1/4 - \eta$ with $\eta ' > \eta$
makes possible the construction. Here $\eta'$ can takes all values in the interval $]0, 1/4[$. The two extreme cases are given by our case 1 and case 2 and this remark yields situations that interpolate between them. The price to pay for these new cases is to combine the two pinching conditions, strong and relative.
\end{Rem}
 
 \begin{Rem}In the proof of the above proposition, the following fact, which will be useful later, was applied several times.
 \begin{claim}\label{claim}
 For every $\epsilon>0$, and $d>0$, there exists $N$ such that for every 
 $\xi  \in \partial \tilde M$, every pair of points
 in a horosphere $H_\xi$ such that $d_{H_\xi}(x,y) \leq d$, then
 $$
 \| \Pi ^\xi (x,y) - \Pi ^\xi _N (x,y)\| \leq \epsilon,
 $$
 where
$\Pi ^\xi _N (x,y)$ is defined in (\ref{approx-N}) with $s=0$.
\end{claim}
Let us prove the claim. In the case of strong $1/4$-pinching, it follows by (\ref{approx-N}) 
$$
\Pi ^{\xi } (x,y) - \Pi ^{\xi}_N (x,y)= \sum _{j=N}^\infty \left(\Pi^\xi _{j+1} (x,y) - \Pi^\xi _{j} (x,y)\right).
$$
and by (\ref{summand-estimate}) we obtain
$$ 
 \| \Pi ^{\xi} (x,y) - \Pi ^{\xi}_N (x,y) \|  \leq  C \sum_{j=N}^\infty  e^{-\tau (t_0 +j)}d_{h_s}(x,y).
$$
This concludes the proof of the claim since the rest of the series satisfies
$$\sum_{j=N}^\infty  e^{-\tau (t_0 +j)}d_{h_s}(x,y) \leq d \sum_{j=N}^\infty  e^{-j \tau} 
\leq \epsilon$$ 
for $N$ large enough.
In the case of relative $1/2$-pinching, the proof is similar, replacing 
(\ref{summand-estimate}) by (\ref{phi-norm-square3}) in the last step.
\end{Rem}

We now wish to compare the stable holonomy with the parallel transport of the Levi-Civita connection
on horospheres. Consider two points $x,y$ on a horosphere $H_\xi$ in $\tilde M$ centered at $\xi \in \partial \tilde M$.
Assume that $d_{H_\xi} (x,y) <\rho$ is smaller than the injectivity radius of $H_\xi$. We recall that,
by Proposition \ref{one-all} (3), the injectivity radius of every horosphere is bounded below by a constant $\rho >0$.
 The stable holonomy
$\Pi ^\xi (x,y)$ and the parallel transport $P^\xi (x,y)$ along the unique geodesic segment joining $x$ and $y$ 
a priori do not coincide. We insist on the fact that the stable holonomy is a dynamical object whereas the Levi-Civita connection is geometric. Assuming that they coincide locally on a horosphere has the following strong implication.
 \begin{Pro}\label{levi-civita-flat}
Let $M$ be a closed Riemannian manifold with sectional curvature satisfying either the strong $1/4$-pinching or relative $1/2$-pinching assumption.
%Let $\tilde v$ be a unit vector tangent to $\tilde M$.
Let $\xi$ be a point in $\partial \tilde M$ and $x_0 \in H_\xi$ be a point in a horosphere centered at $\xi$. Assume that
for every $x,y \in B_{H_\xi} (x_0, \frac{\rho}{2})$, the stable holonomy $\Pi ^\xi (x,y)$ coincide with the parallel transport
$P^\xi (x,y)$ of the Levi-Civita connection of $H_\xi$. Then the induced metric on $H_\xi$ restricted to $B_{H_\xi} (x_0, \frac{\rho}{2})$ is flat.
\end{Pro}

\begin{proof}
Since any pair of points in $B_{H_\xi} (x_0, \frac{\rho}{2})$ are at distance less than $\rho$, there is a unique
geodesic segment joining them and by our coincidence assumption and assertion (2) of \ref{def-stable-holonomy2} it follows that
$$
P^\xi (x,y) = P^\xi (z,y) \circ P^\xi (x,z). 
$$ 
From the classical formula of the curvature in terms of the parallel transport, see for instance \cite[Theorem 7.1]{Lee}, we deduce that the curvature of the induced metric of $H_\xi$ restricted to $B_{H_\xi} (x_0, \frac{\rho}{2})$ is identically zero.

\end{proof}
The goal of what follows is to show that if the stable holonomy and the parallel transport of the Levi-Civita
connection locally coincide on a given horosphere $H_\xi$, then the same property holds on all horospheres. To accomplish this, we need to establish the continuity of the stable holonomy. Let $\tilde v$ be a unit vector tangent to $\tilde M$ and 
$\tilde {v}_k \in T^1 \tilde M$ a sequence of unit tangent vectors such that $\lim_k \tilde v _k = \tilde v$. Let
$\xi _{\tilde v} = c_{\tilde v} (+\infty)$ the associated point on $\partial \tilde M$. Denote by $H_{\tilde v}$ be 
the horosphere centered at $\xi _{\tilde v}$ passing through the base point of $\tilde v$.
Let $\tilde Q_k$ and $\tilde Q$ the lifts of the plaques 
$Q_k$ and $Q$ of the strong stable foliation $W^{ss}$ embedded in a chart $U \subset T^1 M$ and containing 
$\tilde v_k$ and $\tilde v$ respectively. Recall that, from Proposition \ref{one-all}, the sequence of diffeomorphisms
\begin{equation}\label{sequence}
\pi ^{-1} \circ p\circ \Theta (v_k) : D^n \to \tilde p (\tilde{Q}_k)
\end{equation}
converges in the $C^r$-topology to 
\begin{equation}\label{limit}
\pi ^{-1} \circ p\circ \Theta (v) : D^n \to \tilde p (\tilde{Q}).
\end{equation}
\begin{Pro}\label{stableholonomy-continuous}
Let $\tilde {v}_k \in T^1 \tilde M$ be a sequence of unit tangent vectors such that $\lim_k \tilde v _k = \tilde v$.

Let $x= \pi ^{-1} \circ p\circ \Theta (v) (q_x)$, $y= 
\pi ^{-1} \circ p\circ \Theta (v) \mathcal (q_y)$ 
be a pair of point in $\tilde p(\tilde{Q})$ and $x_k = \pi ^{-1} \circ p\circ \Theta (v_k) (q_{x_k})$, 
$y_k = \pi ^{-1} \circ p\circ \Theta (v_k) (q_{y_k})$ in $\tilde p (\tilde{Q}_k) $. 
Then 
$$
\lim _k \Pi^{\xi _{\tilde v_k}} (x_k, y_k) = \Pi ^{\xi _v} (x,y).
$$
\end{Pro}
\begin{proof}
Let us fix $\epsilon >0$. 
By the claim \ref{claim}, we can choose $N$
such that for every $x,y \in \tilde p(\tilde{Q})$
and every $x_k, y_k \in \tilde p (\tilde{Q}_k) $, we have
\begin{equation}\label{approx-N-x-y}
|| \Pi ^{\xi _{\tilde v}} (x,y) - \Pi ^{\xi _{\tilde v}}_N (x,y)|| \leq \epsilon
\end{equation}
and similarly,
\begin{equation}\label{approx-N-x-y-k}
|| \Pi ^{\xi _{\tilde v_k}} (x_k,y_k) - \Pi ^{\xi _{\tilde v_k}}_N (x_k,y_k)|| \leq \epsilon.
\end{equation}
By the above convergence of (\ref{sequence} ) to (\ref{limit}), the points $x_k$ and $y_k$ converge 
to $x$ and $y$ and the unit normals to $\tilde p (\tilde Q _k)$ at $x_k$ and $y_k$ converge to the 
unit normals to $\tilde p (\tilde Q)$ at $x$ and $y$, respectively.
Therefore the flows $(\varphi ^{\xi _{\tilde v_k}}_{t})_{| \tilde p (\tilde Q _k)}$ converge to 
$(\varphi ^{\xi _{\tilde v}}_t)_{| \tilde p (\tilde Q)}$ uniformly for $t\in [0, T]$ for every $T$. Now, the way $\Pi ^{\xi _{\tilde v}}_N (x,y)$ depends on
$\varphi ^{\xi _{\tilde v}}_t$, $t\leq N$ and the fact that $t_0 \leq \log \rho$ implies that 
$\Pi ^{\xi _{\tilde v_k}} _N (x_k,y_k) $ converges
to $\Pi ^{\xi _{\tilde v}}_N (x,y)$. Therefore, there exists $K>0$ such that for all $k\geq K$, 
$$
\| \Pi ^{\xi _{\tilde v}}_N (x,y) - \Pi ^{\xi _{\tilde v_k}} _N (x_k,y_k) \| \leq \epsilon.
$$
We then deduce that for $N$ and $k \geq K$, 
\begin{eqnarray*}
&\| \Pi ^{\xi _{\tilde v}} (x,y) - \Pi ^{\xi _{\tilde v_k}}  (x_k,y_k) \| \leq \\
&\| \Pi ^{\xi _{\tilde v}} (x,y) - \Pi ^{\xi _{\tilde v}}_N (x,y)\| +
\|\Pi ^{\xi _{\tilde v}}_N (x,y) - \Pi ^{\xi _{\tilde v_k}} _N (x_k,y_k) \| +
\| \Pi ^{\xi _{\tilde v_k}} _N (x_k,y_k) - \Pi ^{\xi _{\tilde v_k}}  (x_k,y_k) \|
\end{eqnarray*}
thus,
$\| \Pi ^{\xi _{\tilde v}} (x,y) - \Pi ^{\xi _{\tilde v_k}}  (x_k,y_k) \| \leq 3\epsilon$,
which concludes the proof.

\end{proof}

\bigskip
We can now state the main result of this section.
\begin{Pro}\label{local-global}
Let $M$ be a closed Riemannian manifold with sectional curvature satisfying either the strong $1/4$-pinching or the $1/2$-relative pinching assumption.
%Let $\tilde v$ be a unit vector tangent to $\tilde M$.
Let $\tilde v$ be a unit tangent vector in $T^1 \tilde M$ and $\xi _{\tilde v} = c_{\tilde v} (+\infty)$ the corresponding point in $\partial \tilde M$. Assume that the stable holonomy
$\Pi ^{\xi _{\tilde v}} (x,y)$ and the parallel transport for the Levi-Civita connection $P^{\xi _{\tilde v}} (x,y)$ coincide 
on every ball of radius $\rho /2$ of the horosphere $H_{\xi _{\tilde v}}$. Then for every horosphere $H_{\xi _{\tilde w}}$, $\tilde w \in T^1 \tilde M$, and every $z\in H_{\xi _{\tilde w}}$, there exists  a neighbourhood $\mathcal V (z) \subset H_{\xi _{\tilde w}}$
of $z$ such that the stable holonomy $\Pi ^{\xi _{\tilde v}} (x,y)$ and the parallel transport for the Levi-Civita connection $P^{\xi _{\tilde v}} (x,y)$ coincide for all points $x,y \in \mathcal V (z)$.
\end{Pro}
\begin{proof}
 Suppose that $H_{\tilde v}$ satisfies the assumption of the proposition and let us consider a different horosphere $H_{\tilde w}$. We will prove that locally around $\tilde p (\tilde w)$ on $H_{\tilde w}$, the stable holonomy and the Levi-Civita parallel transport coincide. As mentioned
 in the proof of assertion (2) in Proposition \ref{one-all}, each leaf of
the strong stable foliation $W^{ss}\subset T^1 M$, in particular $W^{ss}(v)$, is dense in $T^1M$, where $v = d\tilde \pi (\tilde v)$.
Moreover, thanks to (\ref{diffeo-tilde-horo}) and (\ref{diffeo-tilde-horo-lim}) in Proposition \ref{one-all},
the lift $\tilde p (\tilde Q) \subset H_{\tilde w}$ 
is the $C^r$ limit
of the sequence of sets $\tilde p (\tilde Q_l)$ where $\tilde Q_l$ are lifts of $Q_l$. These lifts $\tilde Q_l$ 
are subsets of translates, by elements of the fundamental group of $M$, of $H_{\tilde v}$. By the $\pi_1 (M)$-equivariance 
of the stable holonomy (coming from Proposition \ref{parallel-transport}) and of the Levi-Civita 
connection, we get from our assumption that the stable holonomy and the parallel transport of the Levi-Civita connection
coincide on $\tilde p (\tilde Q _l)$. The proof then follows from the continuity properties of Proposition \ref{stableholonomy-continuous}
and Proposition \ref{one-all} (4).

\end{proof}
\begin{Cor}\label{Cor}
Let $M$ be a closed Riemannian manifold with sectional curvature satisfying either the strong $1/4$-pinching or relative $1/2$-pinching assumption. 
If the stable holonomy and the parallel transport of the induced Levi-Civita connection coincide 
on every ball of radius $\rho /2$ of one horosphere $H_{\xi _{\tilde v}}$, then the induced metric on
each horosphere of $\tilde M$ is isometric to a Euclidean metric. Moreover, for every $\tilde w \in T^1 \tilde M$, $x,y \in H_{\xi _{\tilde w}}$ 
we have $\Pi ^{\xi _{\tilde w}}_s (x,y) = P ^{\xi _{\tilde w}}_s (x,y)$, in other words, the stable holonomy and the parallel
transport associated to the Euclidean metric coincide on every horosphere. In particular, the parallel transport associated to 
the Euclidean metric is invariant by the geodesic flow.
\end{Cor}
\begin{proof}
By the Proposition \ref{local-global}, for every horosphere $H_{\xi _{\tilde w}}$ and $x\in H_{\xi _{\tilde w}}$, the stable holonomy and the parallel transport associated to the Levi-Civita connexion
coincide on a neighbouhood $\mathcal V (x)$ of $x$. Thanks to the proposition \ref{levi-civita-flat} applied to 
$\mathcal V (x)$, we deduce that the induced metric on every horosphere has a flat Levi-Civita connexion, hence
is a Euclidean metric. This proves the first part. Let us prove the second part of the Corollary. Let us consider 
$x, y \in H_{\xi _{\tilde w}}$. Choose a continuous path $c : [0, 1 ] \to H_{\xi _{\tilde w}}$ such that $c(0)= x$ and $c(1)=y$.
There exists $t_0 = 0 < t_1 < ... < t_ {2k} = 1$ such that $\{\mathcal V (c(t_{2i}))\}_{i=0}^k$ is a finite covering of $c([0,1])$ and 
$c(t_{2i+1}) \in \mathcal V (c(t_{2i})) \cap \mathcal V (c(t_{2(i+1)}))$. Since the Levi-Civita connexion on the metric of
$H_{\xi _{\tilde w}}$ is flat we have, 
$$
P ^{\xi _{\tilde w}}_s (x,y) = 
P ^{\xi _{\tilde w}}_s (c(t_0), c(t_1)) \circ P ^{\xi _{\tilde w}}_s (c(t_1,c(t_2)) \circ ..... \circ P ^{\xi _{\tilde w}}_s (c(t_{2k-1}, c(t_{2k})) 
$$
and similarly, thanks to the property (2) of the definition \ref{def-stable-holonomy2},
$$
\Pi ^{\xi _{\tilde w}}_s (x,y) = 
\Pi ^{\xi _{\tilde w}}_s (c(t_0), c(t_1)) \circ \Pi ^{\xi _{\tilde w}}_s (c(t_1,c(t_2)) \circ ..... \circ \Pi ^{\xi _{\tilde w}}_s (c(t_{2k-1}, c(t_{2k})). 
$$
We then conclude that $P ^{\xi _{\tilde w}}_s (x, y) = \Pi ^{\xi _{\tilde w}}_s (x, y)$ since $P ^{\xi _{\tilde w}}_s (c(t_j), c(t_{j+1})) = \Pi ^{\xi _{\tilde w}}_s (c(t_j), c(t_{j+1})) $.

\end{proof}

\section{A quasi-isometry between $\tilde M$ and a Heintze group}\label{quasi}

In this section, the main theorem of this article, Theorem~\ref{main-thm}, will be proved. 
As we explained in the introduction, the proof amounts to proving Theorem~\ref{eigenvalues}.
Henceforth, $M$ is assumed to satisfy either the strong $1/4$-pinching 
or relative $1/2$-pinching assumption and to be of dimension greater than or equal to $3$. Furthermore, by Corollary 2.27 we may assume that all the horospheres in $\tilde M$ are isometric to the Euclidean space and that the associated parallel transport is invariant 
by the geodesic flow. We will therefore be able to prove below the following. Given a 
geodesic $c_{\tilde v}(t)$ in $\tilde M$ which projects to a closed geodesic in $M$, there exists a quasi-isometry between 
the universal cover $\tilde M$ of $M$ and the Heintze group $G_A$, where $A$ is the derivative of the first return Poincar\'e map along the closed geodesic. Theorem \ref{G_A-hyperbolic}, will then imply that the eigenvalues of $A$ all have the same modulus,
hence concluding the proof of Theorem \ref{eigenvalues}.

\smallskip

Let us choose a geodesic $c_{\tilde v}(t)$ in $\tilde M$ with end point $\xi = c_{\tilde v} (\infty) \in \partial \tilde M$,  which projects to a closed geodesic in $M$. We consider the horosphere $H_\xi (0)$ centred at $\xi$ and passing through the base point  $x_0= c_{\tilde v}(0)$. For each $p \in \tilde M$,  the geodesic $c$ joining $p$ and $\xi$ intersects $H_\xi (0)$ at a point $x=c(0)$. The pair,
$(t,x)\in \mathbb R \times H_\xi (0)$, are the {\sl horospherical coordinates} of $p$. 

\smallskip

Keeping the same notation as in Section~\ref{connection}, we recall that $\{ \varphi _t\}_{t\in \mathbb{R}}$  is a one parameter group of diffeomorphisms of $\tilde M$ which sends  $H_\xi (0)$ diffeomorphically onto  $H_\xi (t)$ (see \ref{def-phi_t}) and the above horospherical coordinates realise the following diffeomorphism $\Phi : \mathbb R \times H_\xi (0) \rightarrow \tilde M$ defined by
\begin{equation}
\label{eq:preppullback}
(t,x)\rightarrow\varphi_t(x), \text{ for } t\in \mathbb{R} \text{ and } x\in H_\xi (0).
\end{equation}
Therefore, in horospherical coordinates, the pulled back by $\Phi$ of the metric $\tilde g$ on $\tilde M$ at $(t,x)$ writes as the orthogonal sum:
\begin{equation}
\label{metric-horos-coord}
\Phi ^* (\tilde g) = dt^2 + \varphi _t ^* h_t(x),
\end{equation}
where $\varphi _t ^* h_t$ is a flat metric on $H_\xi (0)$. Note that $\varphi_t$ acts by translation on geodesics, hence, there is no effect on the $dt^2$ factor.

\medskip

As before, since the horosphere $(H_\xi (0),h_0)$ is flat  we will identify it with the
 Euclidean space $(\mathbb R ^n, h_{\mathrm{eucl}})$. The geodesic $c_{\tilde v}$ projects to a closed geodesic on $M$ of period $l$.
 Let $\gamma$ be the element of the fundamental group of $M$ with axis $c_{\tilde v}$ such that 
 $D\gamma (\tilde g _l (\tilde v)) = \tilde v$. 
 The map $\psi = \gamma  \circ \varphi _l$ is a diffeomorphism of $\tilde M$, (see definition \ref{psi} below).
When restricted to $H_\xi (0)$, $\psi$ can be considered as a diffeomorphism of $\mathbb R^n$ 
fixing $x_0$, and $d\psi (x_0)$ as a 
linear operator of $\mathbb R^n$ which we will denote by $T$, see the definitions \ref{psi} and \ref{Tk} below,
where $T= T^1$. Up to replacing $T$ by $T^2$,
we can assume that $T$ is contained  in a one parameter group
in $GL(n,\mathbb R)$, i.e. $T=e^{lA}$ for some matrix $A$ (see \cite{cul}). 
Indeed, replacing $T$ with $T^2 = D\psi (x_0) ^2$, we simply work with twice the periodic orbit of period $2l$
and the argument is rigorously the same. We thus can assume from now on that $T= e^{lA}$.
Let us consider the Heintze group $G_A$ associated to the matrix $A$ and 
recall from Section~\ref{horospheres} that 

$G_A = \mathbb R \ltimes _A \mathbb R^n$ is the solvable group 
endowed with the multiplication law
\begin{equation}\label{group-law}
(s,x) \cdot (t,y) = (s+t, x+ e^{-sA} y), \text{ for all } s,t \in \mathbb R, x,y\in\mathbb{R}^n.
\end{equation}

\smallskip

The group $G_A$ is diffeomorphic to $\mathbb R \times \mathbb R^n$, and the tangent
space at each point $(s,x)$ of $G_A$ splits as $\mathbb R \times \mathbb R^n$.
Let us consider the left invariant metric $g_A$ on $G_A$ which is defined to be the standard Euclidean metric
at $(0,0) \in G_A$, where $\mathbb R \times \{ 0 \}$ is orthogonal to $\{0\} \times \mathbb R^n$.
Since the inverse of the left multiplication is given by  $L_{(s,x)^{-1}}(t, y)=(t-s, -e^{sA} x +e^{sA} y)$, an easy computation shows that the metric $g_A$ is then defined, for a vector $Z=(a,X)$ which is tangent to $G_A$
at an arbitrary point $(s,x) \in G_A$, by
\begin{equation}\label{g_T}
g_A (s,x)(Z,Z) = a ^2 + h_{\mathrm{eucl}} (e^{sA} X, e^{sA} X).
\end{equation}

We start by identifying the flat horosphere ($H_\xi (0), h_0)$ with the Euclidean space $(\mathbb R ^n, d_{\mathrm{eucl}})$. 
Let us recall that $c_{\tilde v}$ is a geodesic in $\tilde M$ with $\xi = c_{\tilde v} (\infty) \in \partial \tilde M$,  which projects to a closed geodesic in $M$ of period $l$. 
We do not require that this geodesic is primitive; in fact, we 
will later replace the corresponding element $\gamma$ of the fundamental group by a large enough power of it.

We now consider the diffeomorphism of  $H_\xi (0)$ defined by
\begin{equation}
\label{psi}
\psi (x) =  \gamma \circ \varphi _l (x),\ \text{for } x\in H_\xi (0).
\end{equation}

For all $k\geq 1$, let  $\psi ^k = \psi \circ \psi \dots \circ \psi $ denote the $k$-th power of $\psi$.
For 
$x\in H_\xi (0)$, we define 
\begin{equation}\label{tk}
T_k(x) = d\psi   \left(\psi ^{k-1} (x)\right),
\end{equation}
and 
\begin{equation}\label{Tk}
T^k(x) = T_k(x) \cdot T_{k-1}(x) \cdots T_1(x).
\end{equation}

\begin{figure}[htbp]
\begin{center}
\scalebox{.50}{ 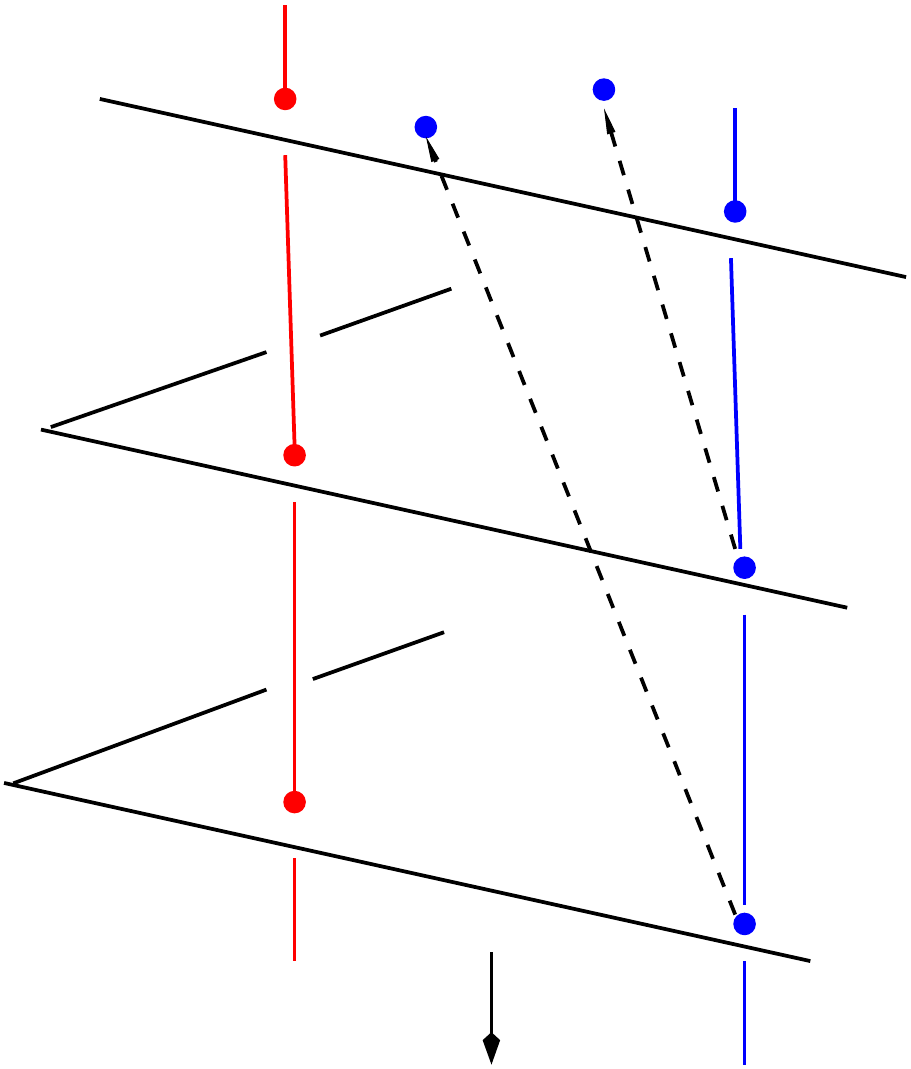}
 \caption{The action of $\psi$ on horospheres.}
\label{figure:quad}
\end{center}
\end{figure}

Since $\gamma$ and $\varphi_t$ commute for all $t\in\mathbb{R}$, it follows  that 
\begin{equation}
\label{eq:T^k}
\psi ^k(x)=\gamma^{k}\circ \varphi_{kl}(x)\text{ and }  T^k(x) = D\psi ^k (x) = D\gamma ^{k} \circ D\varphi _{kl} (x).
\end{equation}

As explained at the beginning of the section, we recall that $T^1 (x_0) = e^{A}$ for $A$ being a $(n\times n)$-matrix. In particular, 
\begin{equation}\label{formula}
T^k(x_0) = D\psi ^k (x_0) = D\gamma ^{k} \circ D\varphi _{kl} (x_0) = e^{lkA}.
\end{equation}

The main result of this section is the following: 
\begin{Thm}
\label{quasi-isom}
With the notation above, $(\tilde M , \tilde g)$ is bi-Lipschitz diffeomorphic, hence quasi-isometric, to $(G_A , g_A)$.
\end{Thm}
\noindent{\it Proof.}
In fact, we will show that there is a bi-Lipschitz diffeomorphism between $G_A$ and $\tilde M$.
Recall that the map $\Phi : \mathbb R \times H_\xi (0) \rightarrow \tilde M$ defined by $\Phi (s,x) = \varphi _s (x)$ is a diffeomorphism.

By Corollary \ref{Cor}, the horosphere $H_\xi (0)$ endowed with the induced metric from $\tilde M$ is flat, hence, 
$\mathbb R \times H_\xi (0) = \mathbb R \times \mathbb R ^n$ and therefore
we can see $\Phi$
as a diffeomorphism between $G_A$ and $\tilde M$. 
 
\smallskip

We first show that the two metrics $\Phi ^* \tilde g$ and $g_A$
coincide at points with coordinates $(lk,y)$ where $k$ is an integer. 
\begin{Lem}
\label{coincidence-integer}
For every $k\in \mathbb Z$ and $y\in \mathbb R^n$, we have 
$\Phi ^ * \tilde g (lk,y)= g_A(lk,y)$.
\end{Lem}
\noindent{\it Proof.} It is clear that  for tangent vectors of the form 
$Z=(a,0)$, we have $\tilde g (Z,Z) = g_A(Z,Z) = a^2$ at {\it any } point of coordinate $(t,x)$.
Therefore, we now focus on tangent vectors of the type $Z=(0,X)$, where 
$X \in \mathbb R^n$ is a vector tangent to $H_\xi (0)= \mathbb R^n$ at $x$.
By (\ref{g_T}), it suffices to show that 
\begin{equation}
\Phi ^* \tilde g (lk,x)(Z,Z) = h_{\mathrm{eucl}} (e^{lkA }X, e^{lkA} X).
\end{equation}
In fact, it follows from (\ref{metric-horos-coord}) that  
\begin{equation}\label{tildeg}
\Phi ^* \tilde g (lk,x)(Z,Z) = h_{lk} (d\varphi _{lk} (X), d\varphi _{lk} (X)),
\end{equation}
where $d\varphi _{lk} (X)$ is a vector tangent to $H_\xi (lk)$ at $x_{lk} = \varphi _{lk} (x)$, and $h_{lk}$ is the
flat metric of $H_\xi (lk)$.
Note that the tangent vector $X$ can be extended to a constant vector field on $\mathbb R^n$, which we will still denote by $X$. 

\smallskip

Recall (see Section~\ref{connection} that for each integer $k$, $P_{lk}^\xi$ is the parallel transport associated to the flat
metric $h_{lk}$ on $H_\xi (lk)$, and that $x_0=c_v (0)$ is the unique point on $H_\xi (0)$ which lies on the axis of 
$\gamma$. Let us denote by $q_{lk}$  the point $\varphi _{lk} (x_0)$.
We thus have 
\begin{equation}
\label{estimate-h_k}
h_{lk} (d\varphi _{lk} (X), d\varphi _{lk} (X)) = h_{lk} \left(P_{lk} ^\xi (x_{lk} ,q_{lk})(d\varphi _{lk} (X)), P_{lk} ^\xi (x_{lk} ,q_{lk})(d\varphi _{lk} (X))\right).
\end{equation}
By assumption (ii) of Theorem \ref{main-thm}, the parallel transport of the flat metric $h_0=h_{eucl}$ ($h_{lk}$) coincides with
the stable holonomy $\Pi_0 ^\xi$ ($\Pi_{lk} ^\xi$) . In particular, the commutation property (3) of the Definition \ref{def-stable-holonomy2} holds:
\begin{equation}\label{commute}
d\varphi _{lk} (x_0) \circ P^\xi_0 (x,x_0)(X)= P_{lk} ^\xi (x_{lk} ,q_{lk})(d\varphi _{lk} (X)) .
\end{equation}
Note that (\ref{commute}) relies on the fact that the family of parallel transports 
of the Levi-Civita connections coincide with the stable holonomies, hence is invariant by the geodesic flow
  and that it is the only place in the proof where we use it.
We now deduce from (\ref{commute}) that
\begin{equation}\label{estimate-h_k2}
h_{lk} (d\varphi _{lk} (X), d\varphi _{lk} (X)) = h_{lk} \left (d\varphi _{lk} (x_0) \circ P^\xi_0 (x,x_0)(X),d\varphi _{lk} (x_0) \circ P^\xi_0 (x,x_0)(X)\right ).
\end{equation}
Since for every $k$, $\gamma ^k$ is an isometry 
we obtain
\begin{equation}\label{estimate-h_k3}
h_{lk} (d\varphi _{lk} (X), d\varphi _{lk} (X)) =h_{0} \left(d\gamma ^{k} \circ d\varphi _{lk} (x_0) (P^\xi_0 (x,x_0)(X)), d\gamma ^{k} \circ d\varphi _{lk} (x_0)(P^\xi_0 (x,x_0)(X))\right),
\end{equation}
thus, by (\ref{formula}),
\begin{equation}\label{estimate-h_k4}
h_{lk} (d\varphi _{lk} (X), d\varphi _{lk} (X)) =h_{0} \left(e^{lkA} (P^\xi_0 (x,x_0)(X)), e^{lkA} (P^\xi_0 (x,x_0)(X))\right).
\end{equation}
Since $H_\xi (0)$ with the induced metric from $\tilde M$ is identified with $\mathbb R^n$, $h_0$ with the standard Euclidean metric $h_{\mathrm{eucl}}$ and $X$ is a constant vector field, we have
$P^\xi_0 (x,x_0)(X) = X$ and 
\begin{equation}
\label{eq:twmore}
h_{0} \left(e^{lkA} \left(P^\xi_0 (x,x_0)(X)\right), e^{lkA} \left(P^\xi_0 (x,x_0)(X)\right)\right) = h_{\mathrm{eucl}} (e^{lkA} X, e^{lkA} X),
\end{equation}
which implies by (\ref{tildeg}) and (\ref{estimate-h_k4}) that
\begin{equation}
\label{eq:onemore}
\Phi ^* \tilde g (lk,x)(Z,Z) = h_{\mathrm{eucl}} (e^{lkA} X, e^{lkA} X) = g_A (lk,x)(Z,Z),
\end{equation}
which completes the proof of Lemma \ref{coincidence-integer}.
\eop{Lemma \ref{coincidence-integer}}

\medskip

For $t\in \mathbb R$, let $k$ be the integer part of $t/l$. We now compare  $g_A (t,x)$ and $g_A (lk,x)$ at any $x \in \mathbb R^n$. Let us set $\sigma = \frac{t}{l}-k$ with $\sigma \in [0, 1[$. For $Z=(0,X)$, we have 
\begin{equation}\label{integer-t}
g_A(t,x) (Z,Z)=h_{\mathrm{eucl}} (e^{tA} X, e^{tA} X) = h_{\mathrm{eucl}} (e^{l\sigma A} e^{lkA} X, e^{l\sigma A} e^{lkA} X)\,.
\end{equation}
Recall that  $e^{lA}=D (\gamma \circ \varphi_l)(x_0) = D \psi (x_0)$ is a fixed $n\times n$ matrix, so that there exists a constant $C$ such that
$\| e^{\pm l\sigma A} \|^2 \leq C$ for every $\sigma \in [0,1[$. Therefore, we deduce from (\ref{integer-t})

\begin{equation}\label{quasi-g_T}
C^{-1} g_A (lk,x) \leq g_A (t,x) \leq C g_A (lk,x) ,
\end{equation}
for every $lk \leq t < (k+1)l$.
On the other hand, we have 
$$h_{t} (D\varphi _t X, D\varphi _t X) = h_{t} ( D\varphi _{l \sigma}  \circ D\varphi _{lk} X, d\varphi _{l\sigma}  \circ D\varphi _{lk} X)$$ and the same argument as before yields,
\begin{equation}\label{quasitilde-g}
C^{-1} \Phi ^* \tilde g (lk,x) \leq \Phi ^* \tilde g(t,x) \leq C \Phi ^* \tilde g (lk,x)\,.
\end{equation}
Then the relations (\ref{quasi-g_T}), (\ref{quasitilde-g}) and Lemma \ref{coincidence-integer} conclude the proof of Theorem \ref{quasi-isom}.

\eop{Theorem~\ref{quasi-isom}}

\medskip

\begin{Cor}
\label{T=Id}
All the eigenvalues of $T = D\psi (x_0)$ have the same modulus. 

\end{Cor}
\begin{proof}
By Theorem \ref{quasi-isom}, $(G_A, g_A)$ is quasi-isometric to $(\tilde M, \tilde g)$. Since $M$ is closed,  $(\tilde M, \tilde g)$ is quasi-isometric to the  finitely generated group $\pi_1 (M)$  endowed with the word metric, which is therefore a hyperbolic group. We thus deduce that $G_A$ is quasi-isometric to a hyperbolic group and by the theorem \ref{G_A-hyperbolic}, this can occur only if the real part of the complex eigenvalues of $A$ are equal.
Recall that $A$ has been chosen so that either $T = e^{lA}$ or $T^2 = e^{lA}$, where $T = D\psi (x_0)$. We deduce that the eigenvalues of $T$ have the same modulus.

\end{proof}

We are now in position to prove Theorem \ref{eigenvalues}, thus,  completing the proof of Theorem \ref{main-thm}.
\begin{proof}[Proof of Theorem \ref{eigenvalues}]

Theorem \ref{quasi-isom} holds for any choice of a closed geodesic, or equivalently of an element $\gamma$ of the fundamental group of $M$, and
so does Corollary \ref{T=Id}. This implies that for any such choice, the moduli of the complex eigenvalues of $T= D\psi (x_0)$ coincide.

Recall that 
\begin{equation}\label{A1}
D\psi (x_0)= e^{l(v)A} = D\tilde p \, \circ \left( D(\gamma \circ \tilde g _{l(v)}(\tilde v)|E^{ss}(\tilde v) \right) \circ D \tilde p ^{-1},
\end{equation}
so that $D g_{l(v)} |E^{ss}$ and $D\psi (x_0)$ are conjugate matrices, therefore we conclude that the eigenvalues of $D g_{l(v)} |E^{ss}$ have the same modulus. 
\end{proof}

\section{Appendix}
The goal of this appendix is twofold. We first will show that the strong $1/4$-pinching assumption implies the bunching of the stable cocycle of the geodesic flow defined in \cite{KS}. Then we will show that the stable holonomy
on the horospheres is conjugate to the stable holonomy defined on the strong stable leaves of the geodesic flow.
\subsection{Strong $1/4$-pinching and bunching}
Under the assumption $-4(1-\tau) \leq K \leq -1$, the strong stable bundle $P: E^{ss} \rightarrow T^1 M$ is
$C^1$ (see  \cite[page 226]{Hi-Pu-Sh}). We choose a $C^1$-metric on $T^1M$ such that the splitting
$TT^1M = E^{ss} \oplus \mathbb{R} \,Z \oplus E^{su}$ is orthogonal, the generator $Z$ of the geodesic flow
satisfies $|Z| =1$ and the metric on $E^{ss}$ and $E^{su}$ are obtained by pulling back the metric of $M$
by the projection $p: T^1 M \to M$.
We consider now the diffeomorphism $f:= g_1 : T^1 M \rightarrow T^1 M$ 
The linear stable cocycle over $f$ defined as
$F:=(Dg_{1}) _{| E^{ss}}$ is also $C^1$ and satisfies 
\begin{equation}\label{cocycle1}
\|F(v)\| \leq e^{-1} \,\,\, {\rm and} \,\,\, \|(F(v))^{-1} \| \leq e^{2(1-\tau)^{1/2}}.
\end{equation}
With the notations of section 2 in \cite{KS}, denoting $\nu (v) := e^{-1}$, $\hat{\nu} (v) := e^{-1}$ we then have $\|Df (v)\| \leq \nu (v)$,
$\|(Df (v))^{-1}\| \geq (\hat{\nu} (v))^{-1}$ and $|Df (Z)| =1$, hence
\begin{equation}\label{appendixbunch}
\|F(v)\| \|(F(v))^{-1} \| \nu (v) \leq e^{-1} \,e^{2(1-\tau)^{1/2}} e^{-1} <1.
\end{equation}
This inequality $\|F(v)\| \|(F(v))^{-1} \| \nu (v) <1$ coincides with the bunching condition (3.1) of \cite{KS} since 
we can take the H\"older coefficient $\beta =1$ since the strong stable bundle is $C^1$.
\subsection{Conjugation of stable holonomies}
Recall that the map $\Phi : T^1 \tilde M \to \tilde M \times \partial \tilde M$ defined by
$\Phi (\tilde v) = (x,\xi)$ where $x = \pi (\tilde v)$ and $\xi = c_{\tilde v} (+\infty)$ is a homeomorphism. By abuse of notation
we will write $\tilde v = (x, \xi)$. Given $\tilde v = (x, \xi)$, the projection $\tilde p : T^1 \tilde M \to \tilde M$ induces a diffeomorphism
between the strong stable leaf $W^{ss} (\tilde v)$ of $\tilde v$ and the horosphere $H_{\xi} (x)$ centered at $\xi$
and passing through $x$.  In particular, $D\tilde p (\tilde v)$ induces an isomorphism between 
$E^{ss} (\tilde v) = T_{\tilde v} W^{ss} (\tilde v)$ and $T_x H_{\xi} (x)$.
\begin{Lem}
Let $\tilde v = (x, \xi)$ and $\tilde w = (y, \xi)$ be on a same strong stable leaf $W^{ss} (\tilde v) \subset T^{1} \tilde M$ and $H_{\xi}(x)$ the horosphere
centered at $\xi$ and passing through $x$ and $y$. 
Then the stable holonomy $\mathcal{H} (\tilde v, \tilde w)$ on $W^{ss}(\tilde v)$ (resp. $\Pi ^\xi (x,y)$ on $H_{\xi} (x)$)
are conjugate,
$$\mathcal H (\tilde v , \tilde w) = (D\tilde p (\tilde w))^{-1} \circ \Pi ^{\xi} (x,y) \circ D\tilde p (\tilde v).$$
\end{Lem}
\begin{proof}
Define $\mathcal H (\tilde v , \tilde w) = (D\tilde p (\tilde w))^{-1} \circ \Pi ^{\xi} (x,y) \circ D\tilde p(\tilde v).$
The Properties (i), (ii) and (iii) of Definition 4.1 for $\mathcal H (\tilde v, \tilde w)$ are consequences of
the corresponding properties of $\Pi ^\xi (x,y)$ stated in Proposition \ref{parallel-transport}.
As stated in the Proposition 4.2 (c), \cite{KS}, the stable holonomy $\mathcal H$ is uniquely determined by the property
that 
\begin{equation}\label{local-identification}
\| \mathcal H (\tilde v , \tilde w) - I (\tilde v, \tilde w)\| \leq C d(\tilde v, \tilde w),
\end{equation}
 where 
$I(\tilde v, \tilde w)$ is a local identification between $E^{ss}(\tilde v)$ and $E^{ss}(\tilde w)$.
On the other hand, as noticed at the bottom of page 173 of \cite{KS}, holonomies do not depend on
the choice of the local identifications. By defining
$$
I (\tilde v, \tilde w) := (D\tilde p (\tilde w))^{-1} \circ P(x,y) \circ D\tilde p (\tilde v),
$$
where $P(x,y)$ is the parallel transport along $H_\xi (x,y)$, we see that
the property (ii) of Proposition \ref{parallel-transport} implies (\ref{local-identification}).
Therefore, the properties (a), (b) and (c) of \cite{KS} Proposition 4.2 are satisfied, which concludes the proof of this lemma.

\end{proof}

%\bibliography{bibliography}{}
%\bibliographystyle{plain}
\end{document}